\documentclass[11pt]{amsart}

\usepackage[
  margin=1in
]{geometry}

\usepackage{amscd}
\usepackage{amssymb}
\usepackage{amsmath}
\usepackage{amsfonts}
\usepackage{enumerate}
\usepackage[all,cmtip]{xy}
\usepackage{tikz-cd}
\usepackage{hyperref}
\usepackage{xcolor}
\usepackage{adjustbox}
\usepackage{graphicx}
\usepackage[T1]{fontenc}
\usepackage{lmodern}

\newtheorem{thm}{Theorem}[section]
\newtheorem{prop}[thm]{Proposition}
\newtheorem{lem}[thm]{Lemma}
\newtheorem{cor}[thm]{Corollary}

\theoremstyle{definition}
\newtheorem{Def}[thm]{Definition}

\newtheorem{rem}[thm]{Remark}
\newtheorem*{notation}{Notation}

\newcommand{\Ss}{\mathbb{S}}
\newcommand{\Zz}{\mathbb{Z}}
\newcommand{\Ee}{\mathbb{E}}
\newcommand{\Ff}{\mathbb{F}_p}
\newcommand{\THH}{\mathrm{THH}}
\newcommand{\Cc}{C_{I_{n+1}}}
\newcommand{\Cp}{C^{\prime}_{I_{n+1}}}

\makeatletter
\@namedef{subjclassname@2020}{2020 Mathematics Subject Classification}
\makeatother

\subjclass[2020]{55P43}
\keywords{Connes' operator, Hochschild homology, Loday's construction}

\numberwithin{equation}{section}

\title{Iterated Topological Hochschild Homology of $\Ff$ at odd primes}

\author[J. Zha]{Jiaxi Zha}
\address{Department of Mathematics, Nankai University, No.94 Weijin Road, Tianjin 300071, P. R. China}
\email{1093913699@qq.com}

\begin{document}

\begin{abstract}
We compute homotopy groups of $n$-fold iterated topological Hochschild homology of $\Ff$ for $n\geq1$ and odd primes. Identifying this as the Loday construction $T^n\otimes\Ff$, we construct a bar spectral sequence computing $\pi_*(T^n\otimes\Ff)$ and prove that this spectral sequence collapses at the $E^2$-page. Finally, we resolve the multiplicative extensions and determine the ring structure.
\end{abstract}

\maketitle

\section{Introduction}

Let $A$ be an $\Ee_n$ ring spectrum for $n\geq2$. The Dunn additivity implies that topological Hochschild homology (THH) is an $\Ee_{n-1}$ ring spectrum, since $\THH$ is a symmetric monoidal functor; see \cite[Theorem 5.1.2.2]{HA} and \cite[IV.2]{NS18}. Therefore, it is natural to consider its $n$-fold iterated THH
\[\underbrace{\THH\circ\THH\circ\cdots\circ\THH}_{n}(A).\]
The aim of this paper is to compute homotopy groups of $n$-fold iterated THH of $\Ff$ for odd primes. Indeed, the computation of homotopy groups of $\THH(\Ff)$ by B\"okstedt in \cite{Bok} is a starting point for computations in topological Hochschild homology. For every prime $p$, B\"okstedt proved that there is an isomorphism of rings
\[\pi_*(\THH(\Ff))\cong \Ff[\mu]\]
with $|\mu|=2$. This computation also provides the initial input for studying iterated THH of $\Ff$. Another approach to compute $\THH(\Ff)$, based on Thom spectra, was given by Blumberg--Cohen--Schlichtkrull \cite[Theorem 1.5]{BCS10}. They proved that there is an equivalence of spectra
\[\THH(\Ff)\simeq\Ff\otimes\Ss[\Omega S^3].\]
Recently, Bayindir-Moulinos refined this equivalence to an $\Ee_2$-equivalence and extended $\Ff$ to any perfect field of positive characteristic. This $\Ee_2$-equivalence provides an approach to compute $n$-fold iterated THH of $\Ff$ for $n\leq 3$ and any prime $p$. However, we adopt a different approach, as our aim is to compute homotopy groups of $n$-fold iterated THH of $\Ff$ for arbitrary $n$. 

For odd primes, the most extensive computations of the homotopy groups of iterated THH of $\Ff$ are, to our knowledge, due to Veen in \cite{Vee18}. He computed homotopy groups of $n$-fold iterated THH of $\Ff$ for $n\leq p$ when $p\geq5$ and for $n\leq2$ when $p=3$. We follow his spectral sequence and the Loday construction approach. More precisely, we use a different decomposition of $T^n$ from Veen's decomposition to construct a bar spectral sequence and prove that this bar spectral sequence collapses at the $E^2$-page. 

The main input for our spectral sequence is the homotopy groups of higher THH of $\Ff$, which can be defined as the Loday construction $S^n\otimes\Ff$; see \cite[\S2]{Vee18}. For $n\geq1$, define $\Ff$-algebras $B_n$ inductively by 
\[B_1=\Ff[\mu],\qquad B_{n+1}=\mathrm{Tor}^{B_n}(\Ff,\Ff)\]
with $|\mu|=2$. Dundas--Lindenstrauss--Richter proved that there is an isomorphism of $\Ff$-algebras
\[\pi_*(S^n\otimes\Ff)\cong B_n.\]
Now, we can state our main theorem.

\begin{thm}\label{thm:main}
	For $n\geq 1$ and $p$ an odd prime, there is an isomorphism of $\Ff$-algebras
	\[\pi_*(T^n\otimes\Ff)\cong\bigotimes_{\varnothing\neq U\subseteq I_n}B_{|U|},\]
	where $I_n=\{1,\cdots,n\}$.
\end{thm}

\subsection*{Organization} In section \ref{sec2}, we recall basic properties of the Loday construction and the definition of $B_n$ from \cite{Vee18}. We then construction an explicit isomorphism $B_n\cong\pi_*(S^n\otimes\Ff)$, whose existence is established in \cite{DLR18}. In Section \ref{sec3}, we construct a spectral sequence converging to $\pi_*(T^n\otimes\Ff)$ and define several operations on it. Finally, we prove that this spectral sequence collapses at the $E^2$-page and resolve its multiplicative extensions in Section \ref{sec4} and Section \ref{sec5}, respectively.

\subsection*{Conventions}
We work with $\infty$-categories throughout, following \cite{HTT,HA}. Therefore, all categories are understood to be $\infty$-categories unless otherwise specified. Moreover,
\begin{enumerate}
\item We write $\mathcal{S}$ for the category of spaces, and $\mathrm{Sp}$ for the category of spectra. We also write $\mathrm{Map}_{\mathcal{C}}(-,-)$ for mapping spaces in a category $\mathcal{C}$. 
\item We write $\Ff$ for both ordinal ring and the Eilenberg-MacLane spectrum $\mathrm{H}\Ff$.
\item We use $\otimes$ to denote both the smash product of spectra and the derived tensor product in $D(\Ff)$. The intended meaning will be clear from the context.
\item For a prime $p$, we write $P(x)$, $\Lambda(x)$ and $\Gamma(x)$ for the polynomial, exterior, and divided power algebras on the classes $x$ over $\mathbb F_p$, respectively. Moreover, we write $P_p(x)$ for $P(x)/(x^p)$. By an algebra generator, we mean $x$ for $P(x)$, $P_p(x)$ and $\Lambda(x)$, and $\gamma_{p^i}(x)$ for $\Gamma(x)$ with $i\geq0$.
\item We write $P_{I_{n+1}}$ for $\Ff[\mu_i\ |\ i\in I_{n+1}]$.
\item We extend integer indexed notation to finite sets. In this case, an object indexed by a finite set $W$ is understood to be the corresponding object for $|W|$.
\end{enumerate}

\section{Recollection on tha Loday construction and $\pi_*(S^n\otimes \Ff)$}\label{sec2}

In this section, we recall some basic properties of the Loday construction, Veen's admissible words in \cite{Vee18} and the computation of homotopy groups of $S^n\otimes\Ff$ given by Dundas--Lindenstrauss--Richter in \cite{DLR18}. We begin by the Loday construction. 

\begin{Def}[{\cite[Definition 2.6]{RSV22}}]
	Let $X$ be a space and $A$ be an $\Ee_{\infty}$ ring spectrum. The Loday construction $X\otimes A$ is an $\Ee_{\infty}$ ring spectrum given by the colimit
	\[X\otimes A:=\mathrm{colim}\left(X\xrightarrow{A}\mathrm{CAlg}(\mathrm{Sp})\right),\]
	where $A$ is the constant functor taking the value $A$.
\end{Def}

Indeed, this defines a functor $\mathcal{S}\times\mathrm{CAlg}(\mathrm{Sp})\to \mathrm{CAlg}(\mathrm{Sp})$, by \cite[Corollary 4.4.4.9]{HTT} and \cite[Corollary 3.2.3.5]{HA}.

\begin{prop}\label{prop:push}
	Let $A$ be an $\Ee_{\infty}$ ring spectrum. The functor $-\otimes A\colon\mathcal{S}\to\mathrm{CAlg}(\mathrm{Sp})$ preserves colimits.
\end{prop}
\begin{proof}
	It follows immediately from \cite[Example 2.5]{HY17}.
\end{proof}
Therefore, given $X,Y\in\mathcal{S}$. We have $(X\times Y)\otimes A\simeq X\otimes(Y\otimes A)$, since $X\times Y\simeq \mathrm{colim}(X\xrightarrow{Y}\mathcal{S}),$ where $Y$ is the constant functor taking the value $Y$. The same proposition also implies that $S^1\otimes A\simeq A\otimes_{A\otimes A}A\simeq \THH(A)$. Hence, we deduce that $T^n\otimes A$ is equivalent to $n$-fold iterated topological Hochschild homology of $A$. We will use this equivalence to compute homotopy groups of iterated topological Hochschild homology of $\Ff$.

Next, we recall Veen's admissible words.
\begin{Def}[{\cite[Definition 3.1]{Vee18}}]
Given the letters $\mu$, $\varrho$, $\varrho^k$ and $\varphi^k$ for $k \geq 0$, define an \emph{admissible word} to be a word such that:
\begin{enumerate}
    \item[(1)] it ends with the letter $\mu$;
    \item[(2)] if $\mu$ is preceded by a letter, it must be $\varrho$;
    \item[(3)] if $\varrho$ is preceded by a letter, it must be $\varrho^k$;
    \item[(4)] if $\varrho^k$ or $\varphi^k$ is preceded by a letter,
    it must be $\varrho$ or $\varphi^l$ for some $l \geq 0$.
\end{enumerate}
An admissible word is called monic, if it begins with one of the letters $\varrho$, $\varrho^0$, $\varphi^0$ or $\mu$.
\end{Def}

Set the degree of $\mu$ is defined to be $2$. We inductively define the degree of an admissible word by the formulas
\[|\varrho x|=1+|x|, \qquad |\varrho^k x|=p^k(1+|x|), \qquad |\varphi^k x|=p^k(2+p|x |).\]
Moreover, let $B_0=\Ff$ and $B_1=\Ff[\mu]$, and 
\[B_n=\bigotimes_{w\in W_{n,\mathrm{odd}}}\Lambda(w)\otimes\bigotimes_{w\in W_{n,\mathrm{even}}}\Gamma(w),\]
for $n\geq2$, where 
\[W_{n,\mathrm{odd}}=\left\{w\ |\ w\  \text{admissible monic of length}\ n,\ |w|\  \text{is odd} \right\}\] 
and 
\[W_{n,\mathrm{even}}=\left\{w\ |\ w\  \text{admissible monic of length}\ n,\ |w|\  \text{is even} \right\}.\]
Using \cite[Proposition 2.2]{BAK+15} and above notations, there is an isomorphism
\[\mathrm{Tor}^{P_p(x)(\Ff,\Ff)}\cong \Lambda(\varrho x)\otimes \Gamma(\varphi^0x),\] 
when the degree of $x$ is even. Here $|\varrho x|=1+|x|$ and $|\varphi^0(x)|=2+p|x|$. Since $\Gamma(w)$ is isomorphic to $\bigotimes_{i\geq 0} P_p(\gamma_{p^i}(w))$, we deduce that
\begin{equation*}
	\mathrm{Tor}^{B_n}(\Ff,\Ff)\cong \bigotimes_{w\in W_{n,\mathrm{odd}}}\Gamma(\varrho^0w)\otimes\bigotimes_{\substack{w\in W_{n,\mathrm{even},}\\ k\geq 0}}[\Lambda(\varrho\gamma_{p^k}(w))\otimes\Gamma(\varphi^0\gamma_{p^k}(w))].
\end{equation*}
Note that $w\in W_{n,\text{even}}$ must be of the form $\varrho^0(v)$ or $\varphi^0(v)$. Hence $\gamma_{p^k}(w)$ is equal to $\varrho^k(v)$ or $\varphi^k(v)$. Therefore, we deduce that
\begin{equation}\label{eq:torbn}
		\mathrm{Tor}^{B_n}(\Ff,\Ff)\cong B_{n+1},
\end{equation} 
for $n\geq 2$.

We recall the following computation of Dundas--Lindenstrauss--Richter.
\begin{thm}[{\cite[Remark 3.9]{DLR18}}]\label{thm:sn and bn}
	There is an isomorphism of $\Ff$-algebras $\pi_*(S^n\otimes\Ff)\cong B_n$.
\end{thm}

For the later computations, we need to relate this computation to the bar spectral sequences and make an explicit choice of the isomorphism $f_n\colon B_n\to\pi_*(S^n\otimes\Ff)$ for $n\geq1$. When $n=1$, we choose $f_1(\mu)$ to be the B\"okstedt class.  For $n\geq2$, since $S^n$ is the pushout of the span $*\leftarrow S^{n-1}\to *$ and $-\otimes\Ff$ preserves pushouts, there is an equivalence
\[S^n\otimes\Ff\simeq \Ff\otimes_{S^{n-1}\otimes\Ff}\Ff,\]
where we identify $S^n$ as $S^1\wedge S^{n-1}$. Hence there is a bar spectral sequence
\begin{equation}\label{eq:Sn}
	E^2_{*,*}(S^n)\cong\mathrm{Tor}_{*,*}^{\pi_*(S^{n-1}\otimes\Ff)}(\Ff,\Ff)\Longrightarrow\pi_*(S^n\otimes\Ff).
\end{equation}

\begin{cor}\label{cor:BncongSn}
	For $n\geq2$, the spectral sequence \eqref{eq:Sn} collapses at $E^2$-page. Moreover, for $n\geq1$, we can inductively choose a class of isomorphisms of $\Ff$-algebras 
	\[\{f_n\colon B_n\to\pi_*(S^n\otimes\Ff)\}_{n\geq1},\]
	such that $f_{n-1}$ induces an isomorphism of $\Ff$-algebras
	\[B_n\cong\mathrm{Tor}^{B_{n-1}}_{*,*}(\Ff,\Ff)\xrightarrow{(f_{n-1})_*} \mathrm{Tor}^{\pi_*(S^{n-1}\otimes\Ff)}_{*,*}(\Ff,\Ff)\cong E^2_{*,*}(S^n)\cong E^{\infty}_{*,*}(S^n)\cong \mathrm{gr}\pi_*(S^n\otimes\Ff),\]
	which is the associated graded of $f_n$.
\end{cor}
\begin{proof}
	We prove this by an induction on $n$. For $n=1$, we have chosen $f_1(\mu)$ to be the B\"okstedt class. For $n\geq2$, suppose that $f_{n-1}$ has been chosen. Then there is an isomorphism $E^2_{*,*}(S^n)\cong B_n$, which implies that
	\[\sum_{s+t=k}\mathrm{dim}_{\Ff}E^2_{s,t}(S^n)=\mathrm{dim}_{\Ff}(B_n)_k=\mathrm{dim}_{\Ff}\pi_k(S^n\otimes\Ff)<\infty,\]
	by Theorem \ref{thm:sn and bn}. Therefore, the spectral sequence collapses at $E^2$-page and induces an isomorphism of $\Ff$-algebras $\bar{f}_n\colon B_n\to\mathrm{gr}\pi_*(S^n\otimes\Ff)$. It remains to lift this to a map of $\Ff$-algebras. Since $p$ is odd, square of any element of odd degree is $0$. Moreover, for $a\in(\pi_*(S^n\otimes\Ff))_{>0},$ $a^p=0,$ which can be checked in $B_n$ by Theorem \ref{thm:sn and bn}.  Now, for any generator $w\in (B_n)_k$, choose an element $x_w\in\pi_k(S^n\otimes\Ff)$ which represents $\bar{f}_n(w)$. Define $f_n$ by sending $w$ to $x_w$. Since $\mathrm{gr}f$ is an isomorphism and $\mathrm{dim}_{\Ff}E^2_{s,t}(S^n)$ is finite for any $(s,t)$, $f_n$ is an isomorphism of $\Ff$-algebras.
\end{proof}

\begin{notation}
	From now on, we do not distinguish between $w\in B_n$ and $f_n(w)\in\pi_*(S^n\otimes\Ff)$. In particular, we have that the filtration degrees of $\varrho x$, $\varrho^kx$ and $\varphi^kx$ are $1$, $p^k$ and $2p^k$, respectively. 
\end{notation}

Moreover, $B_n$ admits a Hopf algebra structure whose comultiplication is induced by 
\[
\left\{
\begin{alignedat}{2}
\Delta_{B_n}(w)
  &= w\otimes 1 + 1\otimes w,
  &\qquad& w\in W_{n,\mathrm{odd}}, \\[6pt]
\Delta_{B_n}\bigl(\gamma_k(w)\bigr)
  &= \sum_{i+j=k}\gamma_i(w)\otimes\gamma_j(w),
  && w\in W_{n,\mathrm{even}}.
\end{alignedat}
\right.
\]
 On the other hand, the identification $S^n\simeq S^1\wedge S^{n-1}$ and the pinch map of $S^1$ induce a Hopf algebra structure on $\pi_*(S^n\otimes\Ff)$ whose comultiplication is denoted by $\Delta_{S^n}$. For $n\geq 1$, we can further choose $f_n$ to be compatible with respect to thees Hopf algebra structures. To see this, let $V_n=\bigoplus_{w\in W_n}\Ff\{w\}$, where $W_n$ is the union of $W_{n,\mathrm{odd}}$ and $W_{n,\mathrm{even}}$. Recall that the tensor algebra $T(V_n)=\bigoplus_{k\geq0}V_n^{\otimes k}$ admits a Hopf algebra structure whose multiplication is the shuffle product and comiltiplication is given by
 \[\Delta(v_1\otimes\cdots\otimes v_n)=\sum_{i=0}^n(v_1\otimes\cdots\otimes v_i)\otimes(v_{i+1}\otimes\cdots\otimes v_n).\]
It admits a Hopf subalgebra $C(V_n)=\bigoplus_{k\geq0}(V_n^{\otimes k})^{\Sigma_k}$. Note that, as a coalgebra, $C(V_n)$ is the cofree conilpotent graded cocommutative coalgebra cogenerated by $V_n$. Moreover, we can equip $C(V_n)$ with a multiplicative filtration such that the filtration degree of $x\in W_n$ is precisely the filtration degree of the same named element in $B_n$.

\begin{lem}
	For $n\geq2$, there is an isomorphism of filtered Hopf algebras $g_n\colon B_n\xrightarrow{\cong} C(V_n)$.
\end{lem}
\begin{proof}
	Since $p$ is odd, this follows from Example 3.8 and Lemma 3.9 of \cite{BGT+18}.
\end{proof}

\begin{prop}\label{prop:hopfbn}
	For $n\geq1$, $f_n$ can be chosen to be an isomorphism of Hopf algebras. In particular, there is an equation
	\begin{equation}\label{eq:coprod of sn}
		\Delta_{S^n}\circ f_n=(f_n\otimes f_n)\circ\Delta_{B_n}.
	\end{equation}
\end{prop}
\begin{proof}
	Again, we prove this by an induction on $n$. For $n=1$, the B\"okstedt class is primitive for degree reasons. For $n\geq2$, suppose that $f_{n-1}$ has been constructed as an isomorphism of Hopf algebras. By the proof of Corollary \ref{cor:BncongSn}, there is an isomorphism of $\Ff$-algebras
	\[f^{\prime}_n:B_n\xrightarrow{\cong}\pi_*(S^n\otimes\Ff)\]
	whose associated graded is given by $\bar{f}_n=(f_{n-1})_*$. By \S3 and Proposition 8.4 of \cite{Vee18}, $\bar{f}_n$ is an isomorphism of Hopf algebras. 
	
	Consider the map $h=\pi\circ g_n\circ (f^{\prime}_n)^{-1}\colon\pi_*(S^n\otimes\Ff)\to V_n$, where $\pi\colon C(V_n)\to V_n$ is the canonical projection map. Since $f^{\prime}_n$ and $g_n$ are maps of algebras and products of elements of positive lengths have length at least two, we have that
	\begin{equation}\label{eq:h}
	    h(1)=0,\qquad h((\pi_*(S^n\otimes\Ff)_{>0})^2)=0,
	\end{equation}
	for $k>0$. Hence, as we noted before the lemma,  there is an unique map of coaugmented coalgebras $H:\pi_*(S^n\otimes\Ff)\to C(V_n)$ such that $\pi\circ H=h$. It is multiplicative by the same universal property. Indeed, if $m_{S^n}$ and $m_C$ are two multiplications, then
	\[\pi\circ H\circ m_{S^n}= h\otimes\epsilon+\epsilon\otimes h=\pi \circ m_C\circ (H\otimes H),\]
	where $\epsilon$ is the counit of $\pi_*(S^n\otimes\Ff)$. Here the first equation follows from Equation \eqref{eq:h}, and the second follows from the equation $\pi\circ m_{ C}=\pi\otimes\epsilon+\epsilon\otimes\pi$. Therefore, we deduce that $H\circ m_{S^n}=m_C\circ(H\otimes H)$. For $k>0$ and $x\in \pi_k(S^n\otimes\Ff)$, the cofree extension is given by
	\[H(x)=\sum_{i\geq0}h^{\otimes (i+1)}\circ\overline\Delta_{S^n}^{(i)}(x),\qquad H(1)=1,\]
	where $\overline\Delta_{S^n}^{(i)}:\pi_k(S^n\otimes\Ff)\to (\pi_*(S^n\otimes\Ff)^{\otimes {i+1}})_{>0}$ is the $i$-fold reduced comultiplication, with $\overline\Delta_{S^n}^{(0)}=id$. Note that the sum is finite in each total degree. Hence, since $h$ is filtered, so is $H$. Identify $\mathrm{gr} C(V_n)$ with $C(V_n)$ using its associated graded. We see that $\pi\circ\mathrm{gr} H\circ\bar{f}_n=\pi\circ g_n.$ Since both $\mathrm{gr} H\circ\bar{f}_n$ and $g_n$ are coalgebra maps, we deduce that $\mathrm{gr} H\circ\bar{f}_n=g_n$. Hence $\mathrm{gr} H$ is an isomorphism. Finiteness of the filtrations in each total degree implies that $H$ an isomorphism of Hopf algebras. Define  $f_n=H^{-1}\circ g_n$, then $\mathrm{gr}f_n=\bar f_n.$
\end{proof}

Finally, let $Q_*(S^n)$ and $P_*(S^n)$ be indecomposible and primitive elements of $\pi_*(S^n\otimes\Ff)_{>0}$, respectively. The proceeding proposition gives a map $P_*(f_{n+1})\colon P_{*+1}(B_{n+1})\to P_{*+1}(S^{n+1}).$ Since $S^{n+1}\otimes\Ff\simeq (S^n\wedge S^1)\otimes\Ff$, the $S^1$-action gives a map $\sigma_n\colon\pi_*(S^n\otimes\Ff)\to \pi_{*+1}(S^{n+1}\otimes\Ff)$. 
\begin{cor}\label{cor:QP}
	$\sigma_n$ restricts to a map $Q_*(S^n)\to P_{*+1}(S^{n+1})$ such that the following diagram is commutative
	\[\begin{tikzcd}[column sep=large, row sep=large]
	Q_*(B_n) \arrow[r, "\sigma_{B_n}"] \arrow[d, "Q_*(f_n)"'] & P_{*+1}(B_{n+1}) \arrow[d, "P_*(f_{n+1})"] \\
	Q_*(S^n) \arrow[r, "\sigma_n"'] & P_{*+1}(S^{n+1}),
	\end{tikzcd}\]
	where $\sigma_{B_n}$ maps $x\in Q_*B_n$ to $[x]\in P_{*+1}B_{n+1}.$
\end{cor}
\begin{proof}
 	By \cite[Proposition 8.5]{Vee18}, $\sigma_n$ takes $x\in\pi_*(S^n\otimes\Ff)$ to the class $[x]\in E^2_{*,*}(S^{n+1})\cong\mathrm{gr}\pi_*(S^{n+1}\otimes\Ff).$ Since $[xy]=\pm d([x|y])$ for $x,y\in \pi_*(S^n\otimes\Ff)_{>0}$, $\sigma_n(xy)\in F_0(S^{n+1})\cong\Ff.$ Hence $\sigma_n(xy)=0$, as $|xy|+1> 0$. Therefore, $\sigma_n$ factors through $\pi_*(S^n\otimes\Ff)_{>0}/\pi_*(S^n\otimes\Ff)_{>0}^2=Q_*(S^n).$ On the other hand, we have that
 	\[\overline{\Delta}(F_1(S^{n+1})_{>0})\subseteq\sum_{i+j=1} (F_i(S^{n+1})_{>0})\otimes (F_j(S^{n+1})_{>0})=0,\]
 	where $\overline\Delta$ is the reduced comultiplication. Therefore, the image of $\sigma_n$ lies in $P_{*+1}(S^{n+1}).$ By the definition of $\sigma_{B_n}$ and the choice of $f_{n+1}$, we have that
 	\[(P_*(f_{n+1})\circ\sigma_{B_n}-\sigma_n\circ Q_*(f_n))(Q_*(B_n))\subseteq F_0(S^{n+1}).\]
 	Since degrees of elements of $Q_*(B_n)$ is positive and $F_0(S^{n+1})=\Ff$, we deduce that the diagram is commutative. 
\end{proof}

\begin{rem}\label{rem:phix}
 	Recall that for an algebra generator $x\in B_n$, we have that
 	\[\sigma_{B_n}(\overline{x})=[x]=
 	\begin{cases}
 	   \rho x, & |x|\text{ even},\\
 	   \varrho^0x, & |x|\text{ odd},
 	   \end{cases}\]
 	 where $\bar{x}$ is the image of $x$ in $Q_*(B_n)$. Hence, in particular, $\varphi^0(x)$ does not belong to the image of $\sigma_n$. 
\end{rem}

\section{The main spectral sequence for the computation}\label{sec3}
In this section, we construct the main spectral sequence used to compute $\pi_*(T^n\otimes\Ff)$ and analyze two operations on this spectral sequences induced by the pinch map and the suspension map, respectively. First, we construct the main spectral sequence $E^r_{*,*}(n+1)$ in Subsection \ref{sub3.1}. Next, in Subsection 3.2, we study comultiplications induced by $n+1$ pinch maps on $T^{n+1}$. Finally, in Subsection \ref{sub3.3}, we consider suspension maps on the spectral sequence.

\subsection{The construction of $E^r_{*,*}(n+1)$}\label{sub3.1}
Let $T=(T^n\cup [0,1]\cup T^n)/(x_0\sim 0,\ x_1\sim 1),$ where $x_0$ and $x_1$ are base points of two copies of $T^n$, respectively. Take $\bar{x}_0$ to be the base point of $T$. We have the following commutative diagram of pointed spaces
\begin{equation}\label{dig:main}
	\begin{tikzcd}[column sep=large, row sep=large]
	T^n\vee T^n \arrow[r, "\simeq"] \arrow[d, "\nabla"] & T \arrow[r] \arrow[d] & T^n\vee S^1 \arrow[d] \\
	T^n \arrow[r, "id"] & T^n \arrow[r] & T^n\times S^1,
	\end{tikzcd}
\end{equation}
where $\nabla$ is the fold map, and the upper right map restricts to the identity on each copy of $T^n$ and to the quotient map $I\to I/\partial I\cong S^1$ on $I$. Now, both left and right squares are pushout squares, so the outer square is also a pushout square. Therefore, there is an equivalence of $\Ff$-algebras
\[T^{n+1}\otimes\Ff\simeq ((T^n\vee S^1)\otimes\Ff)\otimes_{(T^n\vee T^n)\otimes\Ff}(T^n\otimes\Ff)\]
by Proposition \ref{prop:push}. Note that the $\Ee_{\infty}$-ring map induced by $T^n\vee T^n\xrightarrow{\nabla}T^n$ is just the multiplication of $T^n\otimes\Ff$. Therefore, the equivalence of $\Ee_{\infty}$-rings 
\[T^n\otimes\Ff\simeq |B_{\bullet}(T^n\otimes\Ff,T^n\otimes\Ff,T^n\otimes\Ff)|\]
implies that $T^{n+1}\otimes\Ff$ is equivalent to the geometric realization of
\[B_{\bullet}(n+1):=((T^n\vee S^1)\otimes\Ff)\otimes_{(T^n\vee T^n)\otimes\Ff}B_{\bullet}(T^n\otimes\Ff,T^n\otimes\Ff,T^n\otimes\Ff),\]
where the action of $(T^n\vee T^n)\otimes\Ff$ on $B_{\bullet}(T^n\otimes\Ff,T^n\otimes\Ff,T^n\otimes\Ff)$ is induced by the multiplication action of $T^n\otimes\Ff$ on the left and right $T^n\otimes\Ff$ factors. We denote the resulting spectral sequence by $E^r_{*,*}(n+1)$, which is strongly convergent by \cite[Theorem X.2.9]{EKMM}. 

Equivalently, $B_{\bullet}(n+1)$ can be obtained by applying the functor $-\otimes\Ff$ to the simlpicial spaces $(T^n\vee S^1)\cup_{T^n\vee T^n}B_{\bullet}(T^n,T^n,T^n)$. Collapsing $T^n\vee S^1$ and $(n-1)$-skeleton of each $T^n$ gives a map of simplicial spaces $(T^n\vee S^1)\cup_{T^n\vee T^n}B_{\bullet}(T^n,T^n,T^n)\to *\cup_{S^n\vee S^n}B_{\bullet}(S^n,S^n,S^n)$. Note that $(*\cup_{S^n\vee S^n}B_{\bullet}(S^n,S^n,S^n))\otimes\Ff$ is the bar constrcution $B_{\bullet}(\Ff,S^n\otimes\Ff,\Ff).$ Therefore, we obtain a map of spectral sequences
\begin{equation}
	E^r_{*,*}(n+1)\to E^r_{*,*}(S^{n+1}).
\end{equation}

\begin{notation}
	For convenience, we write $A_n$ for $\pi_*(T^n\otimes\Ff)$ and $M_n$ for $\pi_*((T^n\vee S^1)\otimes\Ff)$ in the rest of this paper. In particular, there is an isomorphism $\pi_*((T^n\vee T^n)\otimes\Ff)\cong A_n\otimes A_n$. Moreover, we denote $\bigotimes_{\varnothing\neq U\subseteq S}B_U$ by $C_S$.
\end{notation}

To understand the action of $\pi_*((T^n\vee T^n)\otimes\Ff)$ on $M_n$, we need to analyze the restrictions of the composite along the top row of the above commutative diagram to the first and second copies of $T^n$. Let $l$ and $r$ denote the maps obtained by applying $-\otimes\Ff$ to these two map, respectively. Note that $T^n\vee S^1\simeq B(\mathbb{Z}^n*\langle a\rangle),$ where $\mathbb{Z}^n$ and $\langle a\rangle$ correspond to fundamental groups of $T^n$ and $S^1$, respectively. Conjugation by $a$ gives a map $c\colon T^n\vee S^1\to T^n\vee S^1.$ By construction, we have that $r\simeq c\circ l$.

\begin{lem}\label{lem:l=r}
	As maps of $\Ff$-algebras, $l_*$ agrees with $r_*$. 
\end{lem}
\begin{proof}
	Considering $T^n\vee S^1\simeq B(\mathbb{Z}^n*\langle a\rangle)$ as a groupoid with one object, we have the following commutative diagram
	\[\begin{tikzcd}[column sep=large, row sep=large]
	* \arrow[r, "a"] \arrow[d, "g"'] & * \arrow[d, "c_a(g)"] \\
	* \arrow[r, "a"'] & *
	\end{tikzcd}\]
	for any $g\in \mathbb{Z}^n$, where $c_a:g \mapsto aga^{-1}$ is the conjugation by $a$. Hence the restrictions of the composite along the top row of the Diagram \eqref{dig:main} to the first and second copies of $T^n$ are equivalent as functors $B(\mathbb{Z}^n)\to B(\mathbb{Z}^n*\langle a\rangle)$.  Applying the functor $-\otimes\Ff$ and taking homotopy groups imply the desired result.
\end{proof}
Since $l_*$ factors as $A_n\cong A_n\otimes \Ff\xrightarrow{id\otimes u} A_n\otimes \THH_*(\Ff)\cong M_n$, so is $r_*$. Consequently, we deduce that
\begin{equation}\label{eq:E2}
	E^2_{*,*}(n+1)\cong \Ff[\mu]\otimes \mathrm{Tor}^{A_n\otimes A_n}_{*,*}(A_n,A_n)\cong \Ff[\mu]\otimes \mathrm{HH}_{*,*}(A_n),
\end{equation}
where $\mathrm{HH}$ is the Hochschild homology over $\Ff$. Note that the bidegree of $\mu$ is $(0,2)$. The following lemma is well-known. However, we include its proof, since we need these isomorphisms to be compatible with operations on $E^r_{*,*}(n+1)$ given below.

\begin{lem}\label{lem:HH(A)}
	Let $A$ be $P(x)$, $P_p(x)$ or $\Gamma(x)$ with $|x|$ even, or $\Lambda(x)$ with $|x|$ odd. There is an isomorphism of bigraded $\Ff$-algebras
	\[f\colon\mathrm{HH}_{*,*}(A)\xrightarrow{\cong}A\otimes\mathrm{Tor}^A_{*,*}(\Ff,\Ff),\]
	where the $a\in A$ has bidegree $(0,|a|)$. Moreover, let $p$ be the map from the Hochschild complex to the bar complex given by
	\[p([a_0|\cdots|a_n])=\epsilon(a_0)[a_1|\cdots|a_n]),\]
	where $\epsilon$ is the argumentation of $A$. Then $f$ can be chosen to satisfy $(\epsilon\otimes id)\circ f=p_*$ and $f([1|a])=1\otimes[a]$, where $a$ is the algebra generator of $A$.
\end{lem}
\begin{proof}
	First, assume that $A$ equals $P(x)$, $P_p(x)$ or $\Lambda(x)$. Write $C_{\bullet}(A)$ for the Hochschild complex and $B_{\bullet}(A)$ for the bar complex. For a fixed $k$, we denote the element $1\otimes\cdots\otimes x\otimes\cdots\otimes1\in C_k(A)$ and $1\otimes\cdots\otimes x\otimes\cdots\otimes1\in A\otimes B_k(A)$ with $x$ in $i$th-factor by $x^{(i)}$ and $x^{[i]}$, respectively. Define the map of $\Ff$-algebras $F_k\colon C_k(A)\to A\otimes B_k(A)$ by $F_k(x^{(0)})=x^{[0]}$ and $F_k(x^{(i)})=x^{[0]}+x^{[i]}$ for $1\leq i\leq k$. This is well-defined, since when $A=P_p(x)$, 
	\[(x^{[0]}+x^{[i]})^p=(x^{[0]})^p+(x^{[i]})^p=0,\]
	and when $A=\Lambda(x)$, 
	\[(x^{[0]}+x^{[i]})^2=(x^{[0]})^2+x^{[0]}x^{[i]}+x^{[i]}x^{[0]}+(x^{[i]})^2=0,\]
	since $|x|$ is odd. A direct computation also shows that $F_*$ commute with face maps and degeneracy maps. Similar, $F^{\prime}_k\colon A\otimes B_k(A)\to C_k(A)$ with $F(x^{[0]})=x^{(0)}$ and $F^{\prime}_k(x^{[i]})=x^{(i)}-x^{(0)}$ for $1\leq i\leq k$ also satisfy these relations. Since $F_{\bullet}$ and $F^{\prime}_{\bullet}$ are inverse to each other. $F_{\bullet}$ is an isomorphism of simplicial $\Ff$-algebras. By \cite[\S1.2.3]{HA}, taking geometric realizations and homotopy groups yields the isomorphism $f\colon\mathrm{HH}_{*,*}(A)\xrightarrow{\cong}A\otimes\mathrm{Tor}^A_{*,*}(\Ff,\Ff)$. The isomorphism for $\Gamma(x)$ follows from the isomorphism $\Gamma(x)\cong\bigotimes_{i\geq0}P_p(\gamma_{p^i}(x))$. 
\end{proof}

The initial case $n=1$ of Theorem \ref{thm:main} was proved by B\"okstedt in \cite{Bok}. In the rest of this paper, we fix $n\geq1$ and  assume that Theorem \ref{thm:main} is true for $k\leq n$. The following corollary follows immediately from the proceeding lemma and Equation \eqref{eq:torbn}.

\begin{cor}\label{cor:E2}
	The $E^2$-page of the above spectral sequence can be identified as
	\[E^2_{*,*}(n+1)\cong B_{\{n+1\}}\otimes \bigotimes_{\varnothing\neq U\subseteq I_n}(B_U\otimes B_{U\cup\{n+1\}})=\Cc.\]
\end{cor}

Therefore, we use this identification and the resulting generators of the $E^2$-page throughout.
\begin{rem}\label{rem:comp}
	Throughout the induction, for each proper subset $V\in I_{n+1}$, we use the $\Ff$-algebra isomorphism $\Phi\colon C_V\cong A_V$ and lifts of $E^2$-classes in Section \ref{sec5}. Proper subset factors are included by the corresponding inclusions. Moreover, the chosen lifts of algebra generators of $B_{n+1}$ are sent to zero by collapsing $i$th circle, for every $0\leq i\leq n+1$, and are sent to the corresponding classes in $E^2_{*,*}(S^{n+1})$ by collapsing the $n$-skeleton. There constructions are compatible with Lemma \ref{lem:HH(A)}.
\end{rem}

\subsection{Pinch maps and comultiplications}\label{sub3.2}
In this subsection, we relate spectral sequences to pinch maps. For $1\leq i\leq n+1$, the pinch map of $i$th $S^1$-factor of $T^{I_{n+1}}$ induces a map $\psi_i\colon T^{I_{n+1}}\to T^{I_{n+1}}\vee_{T^{I_{n+1}\setminus\{i\}}} T^{I_{n+1}}$. Our next goal in this section is to construct a spectral sequence for $\pi_*((T^{I_{n+1}}\vee_{T^{I_{n+1}\setminus\{i\}}} T^{I_{n+1}})\otimes\Ff)$ such that $\psi_i\otimes\Ff$ induces a multiplicative map of spectral sequences.

\begin{prop}\label{prop:pinch}
	For $1\leq i\leq n+1$, the pinch map of $i$th $S^1$-factor of $T^{I_{n+1}}$ induces a multiplicative map of spectral sequences
	\[\Delta_{n+1,i}^r:E^r_{*,*}(n+1)\longrightarrow E^r_{*,*}(n+1,i)\]
	such that $E^r_{*,*}(n+1,i)$ converges to $\pi_*((T^{I_{n+1}}\vee_{T^{I_{n+1}\setminus\{i\}}} T^{I_{n+1}})\otimes\Ff),$ for $r\geq2$. Moreover, 
	\[E^2_{*,*}(n+1,i)\cong \Cc\otimes_{C_{I_{n+1}\setminus\{i\}}}\Cc.\]
\end{prop}
\begin{proof}
	First, for $i\le n$, pinch the $i$th $S^1$ in every cpoy of $T^n$ in $B_\bullet(n+1)$. This induces a map of simplicial $\Ee_{\infty}$-rings $B_{\bullet}(\psi_i)\colon B_{\bullet}(n+1)\to B_{\bullet}(n+1,i),$ where $B_{\bullet}(n+1,i)$ is obtained by replacing every copy of $T^n$ in $B_{\bullet}(n+1)$ by $T^{I_n}\vee_{T^{I_n\setminus\{i\}}} T^{I_n}$. Let $E^r_{*,*}(n+1,i)$ be the spectral sequence associated to $B_{\bullet}(n+1,i)$.
	Since $B_{\bullet}(n+1,i)\simeq B_{\bullet}(n+1)\otimes_{B_{\bullet}(n)}B_{\bullet}(n+1)$, there is an isomorphism $E^2_{*,*}(n+1,i)\cong \Cc\otimes_{C_{I_{n+1}\setminus\{i\}}}\Cc.$
	Moreover, 
	\[|B_{\bullet}(n+1)\otimes_{B_{\bullet}(n)}B_{\bullet}(n+1)|\simeq T^{I_{n+1}}\otimes\Ff\otimes_{T^{I_{n+1}\setminus\{i\}}\otimes\Ff} T^{I_{n+1}}\otimes\Ff.\]
	Hence $B_{\bullet}(\psi_i)$ induces a multiplicative map of spectral sequence with desired properties. 
	
	For $i=n+1$, label two circles in the target of the pinch map by $S^1_1$ and $S^1_2$, and let $a_\alpha$ be the generator of $\pi_1(S^1_\alpha)$, for $\alpha=1,2$. Write $B^\alpha_\bullet(n+1)$ for the copy of $B_\bullet(n+1)$ obtained by replacing $S^1$ by $S^1_\alpha$ for $\alpha=1,2$. Moreover, there are maps $l_\alpha,r_\alpha:T^n\otimes\mathbb F_p\to (T^n\vee S^1_\alpha)\otimes\mathbb F_p$ which are the corresponding copies of $l$ and $r$. Now, consider the bisimplicial object
	\[B_{\bullet,\bullet}(n+1,n+1)=B^1_{\bullet}(n+1)\otimes_{T^n\otimes\mathbb F_p}B^2_{\bullet}(n+1),\]
	where the two structure maps are induced by $r_1$ on the right and by $l_2$ on the left, respectively. Geometrically, these two actions correspond to terminal point and initial point of two intervals, respectively. Let $E^r_{*,*}(n+1,n+1)$ be the spectral sequence associated to $\mathrm{diag}B_{\bullet,\bullet}(n+1,n+1)$. By Lemma \ref{lem:l=r}, $E^2_{*,*}$ is free over $A_n$. Hence 
	\[E^2_{*,*}(n+1,n+1)\cong \Cc\otimes_{A_n}\Cc\cong \Cc\otimes_{C_{I_{n+1}\setminus\{n+1\}}}\Cc\]
	Next, we construct another simplicial $\Ee_{\infty}$-$\Ff$-algebra $D_{\bullet}$ and a span of simplicial $\Ee_{\infty}$-$\Ff$-algebras 
	\[B_\bullet(n+1)\xleftarrow{\ \epsilon\  }D_\bullet\xrightarrow{\ P\ }\mathrm{diag}B_{\bullet,\bullet}(n+1,n+1)\]
	such that $\epsilon$ induces an isomorphism of $E^2$-pages of multiplicative spectral sequences associated to the bar complexs. Then the desired multiplicative map $\Delta^r_{n+1,n+1}$ is given by $P^r\circ(\epsilon^r)^{-1}$, for $r\geq2$. In the fundamental group of $T^n\vee(S^1_1\vee S^1_2)\simeq (T^n\vee S^1_1)\cup_{T^n}(T^n\vee S^1_2)$, we have that $a_1g_1a_1^{-1}=g_2$ for $g\in\mathbb Z^n$, where $g_\alpha$ is the copy of $g$ in the $\alpha$th torus. The inclusion of the torus and the loop $a_2a_1$ define a map of pointed spaces $t\colon T^n\vee S^1\to T^n\vee(S^1_1\vee S^1_2)$. Applying $-\otimes\mathbb F_p$ gives a map of $\Ee_{\infty}$-$\Ff$-algebras
	\[\tau:(T^n\vee S^1)\otimes\mathbb F_p\to B_{0,0}(n+1,n+1).\]
	Since $(a_2a_1)g_1(a_2a_1)^{-1}=a_2g_2a_2^{-1}$, we have that $\tau\circ r=r_2$. Moreover, $\tau\circ l=l_1,$ and $r_1=l_2$ by construction. 
	
	To construct the simplicial $\Ee_{\infty}$-$\Ff$-algebra $D_{\bullet}$, let $J$ be the simplicial set $\Delta^1_1\cup_*\Delta^1_2$ where the terminal point of $\Delta^1_1$ is identified with the initial point of $\Delta^1_2$. The inclusions of the initial point of $\Delta^1_1$ and the terminal point of $\Delta^1_2$ induce a map of simplicial spaces $T^n\vee T^n\to \bigvee_{\sigma\in J_{\bullet}}T^n$, where we consider $T^n\vee T^n$ as a constant simplicial space. Define $D_{\bullet}$ to be
	\[D_{\bullet}:=\left[(T^n\vee S^1)\mathop{\sqcup}_{T^n\vee T^n}\left(\bigvee_{\sigma\in J_{\bullet}}T^n\right)\right]\otimes\Ff.\]
	Note that 
	\[(\bigvee_{\Delta^1}T^n)\otimes\Ff\simeq B_{\bullet}(T^n\otimes\Ff,T^n\otimes\Ff,T^n\otimes\Ff).\]
	Hence collapsing $\Delta^1_2$ into a point induces a map of simplicial $\Ee_{\infty}$-$\Ff$-algebras $\epsilon\colon D_{\bullet}\to B_{\bullet}(n+1)$. By \cite[Lemma 2.2]{AR05}, $\epsilon^r$ is an isomorphism for $r\geq2$. To construct the map $P:D_\bullet\to\mathrm{diag}B_{\bullet,\bullet}(n+1,n+1),$ we also consider underlying spaces. On the factor $T^n\vee S^1$, we use $t$ to send it to $T^n\vee(S^1_1\vee S^1_2)$. Moreover, we send each copy of $T^n$ indexed by a simplex of $\Delta^1_\alpha$  to the corresponding copy in the $\alpha$th bar complex by the identify map, for $\alpha=1,2$. Since $\tau\circ l=l_1,$ and $\tau\circ r=r_2,$ these maps agree with $t$ at initial and terminal points. Hence the maps descend to the pushout and define $P$.
\end{proof}

We also need following properties of $\Delta^r_{n+1,i}$ and $E^r_{*,*}(n+1,i)$ constructed in the proceeding proposition.

\begin{lem}\label{lem:prop of delta}
	For $1\leq i\leq n+1$ and $r\geq2$, $\Delta^r_{n+1,i}$ and $E^r_{*,*}(n+1,i)$ have the following properties.
	\begin{enumerate}
	\item The inclusions of the two copies of $T^{n+1}$ induce maps of spectral sequences 
	\[j^r_\alpha\colon E^r_{*,*}(n+1)\longrightarrow E^r_{*,*}(n+1,i), \qquad \alpha=1,2.\]
	Under the identification
	\[E^2_{*,*}(n+1,i)\cong \Cc\otimes_{C_{I_{n+1}\setminus\{i\}}}\Cc,\]
	they are given by $j^2_1(x)=x\otimes1$ and $j^2_2(x)=1\otimes x$, respectively. Moreover, their images generate $E^2_{*,*}(n+1,i)$.
	\item Collapsing another circle induces a map
	\[\rho_\alpha^r\colon E^r_{*,*}(n+1,i)\longrightarrow E^r_{*,*}(n+1), \qquad \alpha=1,2.\]
	Moreover, we have that
	\[\rho^r_1\Delta^r_{n+1,i}=\rho^r_2\Delta^r_{n+1,i}=id.\]
	\item For $U\subseteq I_{n+1}$, let $i_U:T^U\to T^{i_{n+1}}$ be the corresponding inclusion. If $i\in U$, write $\Delta_{U,i}$ for the pinch map on $i$th circle. Then we have that 
	\[\Delta^r_{n+1,i}\circ i^r_U=(i^r_U\otimes_{i^r_{U\setminus\{i\}}} i^r_U)\circ\Delta^r_{U,i},\]
	Moreover, if $i\notin U$, then we have that
	\[\Delta^r_{n+1,i}\circ i^r_U=j^r_1\circ i_U^r=j_2^r\circ i_U^r.\]
	\item Let $q_{n+1}:T^{I_{n+1}}\longrightarrow S^{n+1}$ be the collapse map, and let
	\[Q_{n+1,i}:T^{I_{n+1}\setminus\{i\}}\times(S^1\vee S^1)\longrightarrow S^{n+1}\vee S^{n+1}\]
	be induced by applying $q_{n+1}$ to each torus. Writing $\Delta_{S^{n+1},,i}$ for the corresponding sphere pinch. Then we have
	\[Q_{n+1,i}^r\circ\Delta^r_{n+1,i}=\Delta_{S^{n+1},i}^r\circ q^r_{n+1}.\]
\end{enumerate}
\end{lem}
\begin{proof}
	The proof follows from constructions in Proposition \ref{prop:pinch}.
\end{proof}

We end this subsection with following lemma, which will be used in later computations. For convenience, set $ \Cp=\bigotimes_{\varnothing\ne U\subsetneq I_{n+1}}B_U$. Moreover, for $1\leq i\leq n+1$, let $d_i:\Cp\to C_{I_{n+1}\setminus\{i\}}$ be the map collapsing each factor containing $i$. Consider $C_{I_{n+1}\setminus\{i\}}$ as a subalgebra of $\Cc$, we have that $d_i(x)=x$ for $x\in C_{I_{n+1}\setminus\{i\}}$ and $\bigcap_{i\in I_{n+1}}C_{I_{n+1}\setminus\{i\}}=\Ff$.

\begin{lem}[compare with {\cite[Lemma 9.6]{Vee18}}]\label{lem:pdesc}
	If $y\in E^2_{*,*}(n+1)=\Cc$ satisfies
	\[q^2_{n+1}(y)=0\qquad \text{and}\qquad \Delta^2_{n+1,i}(y)=y\otimes1+1\otimes y\]
	for $i\in I_{n+1}$, then $y\in \Cp$. Moreover, if the $E^r$-page, $\Delta^r_{n+1,i}$ and $q^r_{n+1}$ are the same as these in $E^2$-page in total degrees at most $|y|$, then $y\in \Cp$ under this identification. Finally, if $A_{n+1}$ contains $\Cp$ as a subalgebra and homogeneous lifts $\widetilde{v}$ of the monomial basis of $B_{n+1}$ along $q_{n+1}$ form a free $\Cp$-basis with $\widetilde{1}=1$, then the same conclusion holds in $A_{n+1}$. 
\end{lem}
\begin{proof}
	On $E^2$-page, write $y=\sum_v a_vv$ with $a_v\in C'_{I_{n+1}}$, where $v$ ranges over the monomial basis of $B_{n+1}$. If $y\notin C'_{I_{n+1}}$, then we can choose some $v\ne1$ of maximal degree with $a_v\ne0$. We claim that $a_v\in\Ff$. Indeed, choose another element $w$ in the monomial basis of $B_{n+1}$. If $|w|=|v|$, then the coefficient of $1\otimes v$ in $\Delta^2_{n+1,i}(w)$ is zero by Lemma \ref{lem:prop of delta}.(4) and the counit identity. The same happens when $|w|<|v|$ for degree reasons. Now, comparing coefficients of $1\otimes v$ in the equation 
	\[y\otimes1+1\otimes y=\Delta^2_{n+1,i}(y)=\sum_v \Delta^2_{n+1,i}(a_v)\Delta^2_{n+1,i}(v)\]
	gives that $\Delta^2_{n+1,i}(a_v)=1\otimes a_v$. By Lemma \ref{lem:prop of delta}.(2), applying $\rho^2_2$ to both sides implies that 
	\[a_v=\rho^2_2\circ\Delta^2_{n+1,i}(a_v)=\rho^2_2(1\otimes a_v)=d_i(a_v)\in C_{I_{n+1}\setminus\{i\}}.\]
	Hence $a_v\in\bigcap_{i\in I_{n+1}}C_{I_{n+1}\setminus\{i\}}=\Ff$. Consequently, the coefficient of $v$ in $q^2_{n+1}(y)=0$ is nonzero, which is a contradiction. The same argument also works for later pages.
	
	Next, assume that $A_{n+1}$ is free over $\Cp$. Then $A_{n+1}$ is also free over $C_{I_{n+1}\setminus\{i\}}$ and the target of $\Delta_{n+1,i}$ is isomorphic to $A_{n+1}\otimes_{C_{I_{n+1}\setminus\{i\}}}A_{n+1}$. Hence the same argument as above also applies in this case.
\end{proof}

\subsection{The suspension maps}\label{sub3.3} Finally, we consider the relationship between the suspension map and $E^r_{*,*}(n+1)$. First, we consider the ordinal suspension operation, which is induced by the equivalence $T^{n+1}\otimes\Ff\simeq \THH(T^n\otimes\Ff)$, where the topological Hochschild homology is induced by the last $S^1$-factor. Denote the resulting map by $\sigma_{T^n}\colon\pi_*(T^n\otimes\Ff)\to \pi_{*+1}(T^{n+1}\otimes\Ff).$ Note that, by construction, we have that $q_{n+1}\circ \sigma_{T^n}=\sigma_{B_n}\circ q_n$.

\begin{lem}\label{lem:susp}
	For $a\in A_n$, $\sigma_{T^n}(a)$ has a representative in $F_1(n+1)$ such that its image in $E^2$-page is represented by $1\otimes [1\otimes a]\in \Ff[\mu]\otimes\mathrm{HH}(A_n)$.
\end{lem}
\begin{proof}
	For convenience, we write $M$ for $(T^n\vee S^1)\otimes\Ff$ in this proof. Then the unit map gives a map $j:=\eta_M\otimes id\colon T^n\otimes\Ff\to M\otimes (T^n\otimes\Ff)=B_1(n+1)$ such that $d_0\circ j=l$ and $d_1\circ j=r$. Hence we obtain a path $\alpha\colon r\Rightarrow l$ in $\mathrm{Map}_{\mathrm{Sp}}(T^n\otimes\Ff,\mathcal{F}_1(n+1))$, where $\mathcal{F}_1(n+1)=|\mathrm{sk}_1B_{\bullet}(n+1)|$ and we denote the image of $r$ and $l$ in $\mathcal{F}_1(n+1)$ by the same notations. Lemma \ref{lem:l=r} gives a path $\beta\colon l\Rightarrow r$ in $\mathrm{Map}_{\mathrm{Sp}}(T^n\otimes\Ff,M)$. Composing with the canonical map $M\to \mathcal{F}_1(n+1)$ gives a path in $\mathrm{Map}_{\mathrm{Sp}}(T^n\otimes\Ff,\mathcal{F}_1(n+1))$, which is still denoted by $\beta$. Consequently, we obtain an element
	\[\alpha*\beta\in\pi_1(\mathrm{Map}_{\mathrm{Sp}}(T^n\otimes\Ff,\mathcal{F}_1(n+1)),r)\cong[\Sigma T^n\otimes\Ff,\mathcal{F}_1(n+1)].\]
	Composing with the canonical map $i_1:\mathcal{F}_1(n+1)\to |B_{\bullet}(n+1)|\simeq T^{n+1}\otimes\Ff$, we see that $i_1(r)=i_1(l)$, since conjugation by $a$ is trivial in $\Zz^n\times \Zz$. In particular, there is a commutative diagram
	\[\begin{tikzcd}[column sep=large, row sep=large]
	B_1(n+1) \arrow[r, shift left=0.8ex, "d_0"] \arrow[r, shift right=0.8ex, "d_1"'] \arrow[d] & B_0(n+1) \arrow[d] \\
	(T^n\times S^1)\otimes\Ff \arrow[r, shift left=0.8ex, "id"] \arrow[r, shift right=0.8ex, "id"'] & (T^n\times S^1)\otimes\Ff,
	\end{tikzcd}\]
	which implies that $i_1(\alpha)$ is a constant path by naturality. By Lemma \ref{lem:l=r}, we see that $i_1(\beta)$ is the circle induced by the natural transformation
	\[\begin{tikzcd}[column sep=large, row sep=large]
	* \arrow[r, "a"] \arrow[d, "g"'] \arrow[d] & * \arrow[d, "g"] \\
	* \arrow[r, "a"] & *
	\end{tikzcd}\]
	of functors $a: i_{\Zz^n}\Rightarrow i_{\Zz^n}$, where $i_{\Zz^n}:B(\Zz^n)\to B(\Zz^n\times\Zz)$ maps the morphism $g$ to $g$. This natural transformation can be obtained by restricting the functor
	\[B(\Zz)\times B(\Zz^n)\longrightarrow B(\Zz^n\times\Zz),\qquad (m,g)\mapsto (g,m)\]
	along the map $\Delta^1\to B(\Zz)$ which sends $0\to1$ to $1$. That is, we can identify $i_1(\beta)$ as the map
	\[(S^1_+)\otimes(T^n\otimes\Ff)\to T^{n+1}\otimes\Ff,\]
	where the first tensor product is formed in $\mathrm{Sp}$. This is the definition of $\sigma_{T^n}$. Consequently, we obtain that $i_1(\alpha*\beta)\simeq i_1(\beta)\simeq\sigma_{T^n}.$
	
	Finally, since the image of $\beta$ factors through $M$, taking the cofiber of $\mathcal{F}_0(n+1)\to\mathcal{F}_1(n+1)$ implies that $\beta$ becomes null after composing with $\pi$, where $\pi$ is the canonical map $\mathcal{F}_1(n+1)\to \mathcal{F}_1(n+1)/\mathcal{F}_0(n+1)$. Since $\mathcal{F}_1(n+1)/\mathcal{F}_0(n+1)\simeq\Sigma\mathrm{cofib}(s_0),$ we see that $\pi(\alpha*\beta)=\Sigma\bar{j}$, where $\bar{j}$ is the composition of $j$ and the canonical map $B_1(n+1)\to \mathrm{cofib}(s_0)$. That is, $(\alpha*\beta)_*(a)$ is represented by $\bar{j}_*(a)$ in $E^1$-page. Passing to $E^2$-page, we obtain that $(\alpha*\beta)_*(a)$ is represented by $1\otimes[1\otimes a]$. Hence taking the representative to be $(\alpha*\beta)_*(a)$ gives the desired result.
\end{proof}

Next, for $1\leq i\leq n+1$, we can also consider the $S^1$-action on $i$th $S^1$-factor of $T^{n+1}$. For this, consider the following map
\begin{equation*}
	\begin{split}
		& S^1\simeq \mathrm{Map}_{\mathcal{S}}(*,S^1) \xrightarrow{-\times T^{n+1}}\mathrm{Map}_{\mathcal{S}}(T^{n+1},S^1\times T^{n+1})\xrightarrow{\tau_i\circ-} \mathrm{Map}_{\mathcal{S}}(T^{n+1},T^{n+2}) \\
		& \xrightarrow{(id_{T^{i-1}}\times m_{S^1}\times id_{T^{n+1-i}})\circ-}\mathrm{Map}_{\mathcal{S}}(T^{n+1},T^{n+1})\xrightarrow{-\otimes\Ff} \mathrm{Map}_{\mathrm{CAlg}(\mathrm{Sp})}(T^{n+1}\otimes\Ff,T^{n+1}\otimes\Ff)
	\end{split}
\end{equation*}
where $\tau_i$ places the first $S^1$-factor to the $i$th $S^1$-factor. Taking its adjoint and using the splitting $\Sigma^{\infty}_+S^1\simeq \Sigma\Ss\oplus\Ss$ give a map $\sigma_{n+1,i}\colon \Sigma(T^{n+1}\otimes\Ff)\to T^{n+1}\otimes\Ff$; compare with \cite[Proposition 2.7]{Vee18}.  Consequently, we obtain a map $\pi_*(T^{n+1}\otimes\Ff)\to\pi_{*+1}(T^{n+1}\otimes\Ff)$, which is still denoted by $\sigma_{n+1,i}$. Note that, for $i\in V\subseteq I_{n+1}$, we have the following commutative diagram
	\[\begin{tikzcd}[column sep=large, row sep=large]
	\Sigma(T^V\otimes\Ff) \arrow[r, "\sigma_{V,i}"] \arrow[d] & T^V\otimes\Ff \arrow[d] \\
	\Sigma(T^{n+1}\otimes\Ff) \arrow[r, "\sigma_{n+1,i}"'] & T^{n+1}\otimes\Ff.
	\end{tikzcd}\]
Write $q_S\colon C_{S}\to B_{S}$ for the map collapsing $(|S|-1)$-skeleton. We need the following lemma to resolve multiplicative extensions. For $1\leq i\leq n+1$, the inclusion $\{i\}\hookrightarrow I_{n+1}$ induces a map $A_{\{i\}}\to A_{n+1}$, and the image of $\mu$ under this map is denoted by $\mu_{i}$. We also use the same notation to denote the class represented by $\mu_i$ in the spectral sequence.

\begin{lem}\label{lem:suspn+1}
	Assume that $A_{n+1}$ satisfies the assumption in Lemma \ref{lem:pdesc}. If $y\in \Cp$ satisfies $\sigma_{n+1,i}(y)=0$ for $ i\in I_{n+1}$, then $y\in P_{I_{n+1}}^p.$
\end{lem}
\begin{proof}
	For each $\varnothing\ne V\subseteq I_{n+1}$, let $x_{V,v}$ be the element of $\Cc$ whose $V$-factor is $v\in B_V$ and other factors are $1$. Here $v$ is $\mu$, if $|V|=1$. If $|V|\geq 2$, then $v$ is either a monic word of length $|V|$ and odd degree or $\gamma_{p^j}(w)$, where $w$ is a monic word of  length $|V|$ and even degree and $j\geq 0$. Note that these are precisely the minimal variables defined in Section \ref{sec2}. In particular, $|x_{V,v}|=|v|$. Now, as an $P_{I_{n+1}}^p$-module, $C'_{I_{n+1}}$ is free with basis
	\[\mathcal{C}=\left\{1\ ,\ c=\prod_{\substack{(V,v\in B_V)\\ \varnothing\neq V\subsetneq I_{n+1}}}x_{V,v}^{e_{V,v}(c)} \ \ |\ \ e_{V,v}(c)\in
	\begin{cases}
	\{0,\ldots,p-1\},&|x_{V,v}|\ \text{even},\\
	\{0,1\},&|x_{V,v}|\ \text{odd}
	\end{cases}\right\}.\]
	Hence we can write $y$ as $\sum_{c\in\mathcal{C}} f_c c$ with $f_c\in P_{I_{n+1}}^p$, where only finitely many $f_c$ are nonzero. Suppose that $y\notin P_{I_{n+1}}^p$. Then $f_c\ne0$ for some $c\ne1$. Among all generators $x_{V,v}$ occurring with positive exponent in $c\neq1$ with $f_c\ne0$, choose $x=x_{U,u}$ such that 
	\[|U|=\max_{\substack{f_c\ne0\\e_{V,v}(c)>0}}|V|,\qquad |x|=\max_{\substack{f_c\ne0,\ e_{V,v}(c)>0\\|V|=|U|}}|x_{V,v}|.\]
	Since $U\subsetneq I_{n+1}$, there is an $i\in I_{n+1}\setminus U$. Set $W=U\cup\{i\}$, then $|W|=|U|+1>|U|\ge1.$ Hence, by the maximality of $|U|$, neither $c$ occurring in $y$ nor its coefficient $f_c$ contains a positive-degree element in the factor indexed by $W$.
	
	Now, since $i\notin U$, we have that
	\begin{equation}\label{eq:u}
		q_W\circ\sigma_{n+1,i}(x)=\pm\sigma_{B_U}\circ q_U(x)=\pm\sigma_{B_U}(u),
	\end{equation}
	where we consider $x_{U,u}$ as an element of $C_W$. When $W\subsetneq I_{n+1}$, we have that
	\[\Cp\cong \left(\bigotimes_{\substack{\varnothing\ne V\subsetneq I_{n+1}\\ V\neq W}}B_V\right)\otimes B_W.\]
	Hence, we can consider the coefficient of $x_{W,\sigma_{B_U}(u)}$ in $\sigma_{n+1,i}(y)$ as an element in $\left(\bigotimes_{\substack{\varnothing\ne V\subsetneq I_{n+1}\\ V\neq W}}B_V\right)$, since $|U|+1<n+1$. If $W=I_{n+1}$, then 
	\[A_{n+1}\cong \bigoplus_v \Cp\widetilde{v}\]
	by the assumption. Hence we can consider the coefficient of $\widetilde{x}_{W,\sigma_{B_U}(u)}$ as an element in $\Cp$. Since applying $\sigma_{n+1,i}$ to a $V$-indexed variable in $c$ gives a $(V\cup\{i\})$-indexed class by definition, it can contribute to the coefficient of $x_{W,\sigma_{B_U}(u)}$ only if $W\subseteq V\cup\{i\}$, which implies $V=U$ by the choice of $U$. Moreover, only elements of degree $|x|$ can contribute and the coefficient lies in $\Ff$, by the choice of $x$ and the equation $|x_{W,\sigma_{B_U(u)}}|=|x|+1$. Choose such an $x^{\prime}$. By Equation \eqref{eq:u}, $\sigma_{B_U}(x^{\prime})=\sigma_{B_U}(x)$. Hence Remark \ref{rem:phix} implies that only $x$ can contribute. Since the suspension map is a derivation by \cite[Proposition 5.10]{AR05}, it kills $P_{I_{n+1}}^p$. Therefore, we have that 
	\[\sigma_{n+1,i}(y)=\sum_{c\in\mathcal{C}}f_c\sigma_{n+1,i}(c).\]
	For a $c$ containing $x$, a direct computation implies that the coefficient of $x_{W,\sigma_{B_U}}(u)$ or $\widetilde{x}_{W,\sigma_{B_U}(u)}$ is $\pm e_{U,u}(c)f_c\frac{c}{x}$, which is nonzero since $1\leq e_{U,u}(c)\leq p-1$. If there is only a single $c$ contains $x$, then $\sigma_{n+1,i}(y)\neq 0$, a contradiction. Otherwise, since $c\in\mathcal{C}$ are basis monomials,
	\[\{\frac{c}{x}\ |\ c\in\mathcal{C}\ \mathrm{contains}\ x\}\]
	are linearly independet over $P_{I_{n+1}}^p$. Since $1\leq e_{U,u}(c)\leq p-1$, the coefficient of $x_{W,\sigma_{B_U}(u)}$ is nonzero, which implies that $\sigma_{n+1,i}(y)\neq 0$, a contradiction.
\end{proof}

Finally, recall that there is an isomorphism $E^2_{*,*}(n+1)\cong \Ff[\mu]\otimes \mathrm{HH}_{*,*}(A_n)$ as Equation \eqref{eq:E2}. Hence we can consider Connes' operator on $E^2$-page.
\begin{prop}\label{prop:connes}
	For $r\geq2$, there are operations $\delta_{n+1}^r:E^r_{s,t}(n+1)\longrightarrow E^r_{s+1,t}(n+1)$ such that
	\[d^r\delta^r_{n+1}=(-1)^r\delta^r_{n+1}d^r,\qquad \delta^2_{n+1}=id_{\Ff[\mu]}\otimes B,\]
	and $\delta^r_{n+1}$ induces $\delta^{r+1}_{n+1}$ on $E^{r+1}$-page, where $B$ is Connes' operator on $\mathrm{HH}_{*,*}(A_n)$. Moreover, under the identification $E^2_{*,*}(n+1)\cong \Cc$, we have that
	\[\delta^2_{n+1}(x_{U,u})=
	\begin{cases}
	x_{U\cup\{n+1\},\,\sigma_{B_U}(u)}, & n+1\notin U,\\[2mm]
	0,& n+1\in U,
	\end{cases}\]
	for $\varnothing\neq U\subseteq I_{n+1}$ and every algebra generator $u$ of $B_U$.
\end{prop}
\begin{proof}
	Recall that $E^r_{*,*}(n+1)$ is induced by applying $-\otimes\Ff$ to the bar construction 
	\[(T^n\vee S^1)\cup_{T^n\vee T^n}B_{\bullet}(T^n,T^n,T^n)\]
	This bar construction can be obtained by applying the classifying spaces functor to the simlpicial group $G_{\bullet}$, where $G_i=\langle a\rangle*(\Zz^n)^{*(i+1)}$, $d_0$ is induced by $g_1\mapsto g_0$, $d_i$ is induced by $g_i\mapsto ag_0a^{-1}$, and other faces and degeneracies are induced by identifying the adjacent factors and inserting factors. Next, for $i\geq0$, define $F_i$ by $F_i(a)=a$ and
	\[F_i(g_j)=\begin{cases}
	g_{j+1},&0\leq j<i,\\
	ag_0a^{-1},& j=i.
	\end{cases}\]
	In particular, $F_0=c_a$, the conjugation by $a$. A direct computation implies that
	\[d_0F_i=d_i,\qquad s_0F_i=F_{i+1}^2s_i,\]
	\[d_jF_i=F_{i-1}d_{j-1},\qquad s_jF_i=F_{i+1}s_{j-1},\]
	for $1\le j\le i$, and $F^{i+1}_i=c_a$, where $c_a$ is the conjugation by $a$.
	
	We use these maps to define Connes' operator for $E^r_{*,*}(n+1)$. First, note that the following maps
	\[h_j(x_1,\cdots,x_k)=(c_a(x_1),\cdots,c_a(x_j),a,x_{j+1},\cdots,x_k)\] 
	for $0\leq j\leq k$, give a homotopy from $c_a:B_{\bullet}(G_i)\to B_{\bullet}(G_i)$ to the identity, by \cite[Definition 5.1]{May68}. Note that maps $h_j$ do not preserve base point, so $h_j\otimes\Ff$ is not an $\Ff$-algebra map. Hence, to write the chain homotopy explicitly, we use the functor $(B_{\bullet}(-)\otimes\Ff)\otimes\Ff$ to pass to the category of $\Ff$-algebras.  Also note that, since $F_i$ are pointed for $i\geq0$, $B_{\bullet}(F_i)\otimes\Ff$ is a map of $\Ff$-algebras, and similarly for face maps and degeneracy maps. By \cite[Theorem 7.1.2.13]{HA}, we can identify $(B_k(G_i)\otimes\Ff)\otimes\Ff$ as a chain complex $(C_{k,i,l},\partial_l)$. We still use $d_j$, $s_j$, $h_j$, $c_a$ and $F_i$ to denote the corresponding maps of chain complexes. Set $b_i=\sum_{j=0}^i(-1)^jd_j$, $s^{\prime}=F_{i+1}\circ s_i$ and 
	\[B_{G_i}=(1+(-1)^iF_{i+1})\circ s^{\prime}\circ\sum_{j=0}^i(-1)^{ij}F_i^j\colon C_{k,i,l}\to C_{k,i+1,l};\]
	see \cite[\S4,5]{Pon20} for a more detailed discussion. A direct computation implies that
	\begin{equation}\label{eq:bB}
		b\circ B_{G_{\bullet}}+B_{G_{\bullet}}\circ b=1-c_a. 
	\end{equation}
	The degeneracy relations imply that $B_{G_{\bullet}}$ descends to the normalized bar complex, whose bar-degree filtration models the skeletal filtration of the geometric realization. Next, set $b^{bar}_k=\sum_{j=0}^k(-1)^j d^{bar}_j,$ where $d^{bar}_j$ is the face map of $B_k(G_i)$ for $0\leq j\leq k$. Also, in degree $(k,i,l)$, set 
	\[d^{int}_t=\partial_l+(-1)^lb^{bar}_k,\qquad H_t=(-1)^l\sum_{j=0}^k(-1)^jh_j,\]
	where we combine the degrees $k$ and $l$ into $t=k+l$. Hence both $d^{int}$ and $H$ preserve $i$, while decreasing and increasing $t$ by one, respectively. Again, a direct computation implies that
	\begin{equation}\label{eq:dH}
		d^{int}\circ H+H\circ d^{int}=1-c_a.
	\end{equation}
	Since $h_j$ commutes with face maps for $0\leq j\leq k$, $H$ also descends to the normalized bar complex, and we have that $b\circ H=H\circ b$. Moreover, the maps defining $B_{G_{\bullet}}$ commutes with $d^{bar}_j$ and $\partial_l$, so $d^{int}\circ B_{G_{\bullet}}=B_{G_{\bullet}}\circ d^{int}$. Now, combining the degrees $k$ and $l$, the differential and the required operator is defined to be
	\[d:=d^{int}+(-1)^tb,\qquad \mathcal{B}:=(-1)^t\circ B_{G_{\bullet}}-H.\]
	The proceeding computation then implies that
	\[d\circ \mathcal{B}(x)=(-1)^t d^{int}\circ \mathcal{B}_{G_{\bullet}}(x)+b\circ \mathcal{B}_{G_{\bullet}}(x)-d^{int}\circ Hx-(-1)^{t+1}b\circ H(x),\]
	and
	\[\mathcal{B}\circ d(x)=(-1)^{t-1}\mathcal{B}_{G_{\bullet}}\circ d^{int}(x)-H\circ d^{int}(x)+\mathcal{B}_{G_{\bullet}}\circ b(x)-(-1)^tH\circ b(x).\]
	Therefore, $d\circ \mathcal{B}+\mathcal{B}\circ d=(b\circ \mathcal{B}_{G_{\bullet}}+\mathcal{B}_{G_{\bullet}}\circ b)-(d^{int}\circ H+H\circ d^{int})=0$ by Equation \eqref{eq:bB} and Equation \eqref{eq:dH}.
	
	Now, the filtration $F_s(n+1)$ is obtained by taking the realization of the bar complex restricted to degrees $i\leq s$. Since $B_{G_{\bullet}}$ increases $i$ by one and preserves $k$ and $l$, while $H$ increases $k$ by one and presserves $l$ and $i$, $\mathcal{B}$ increases filtration by at most one and total degree by one. Write $\mathcal{B}^r$ for the induced maps on the spectral sequence for $F_s(n+1)\otimes\Ff$. The unit map $\Ss\to\Ff$ and the $\Ff$-module structure on $F_s(n+1)$ give maps
	\[F_s(n+1)\xrightarrow{\eta_s} F_s(n+1)\otimes\Ff\xrightarrow{\rho_s} F_s(n+1)\]
	such that $\rho_s\circ\eta_s=id$, for $s\geq0$. It is clear that these maps are compatible with inclusions $F_s(n+1)\to F_{s+1}(n+1)$ for $s\geq0$. Hence they induce maps on spectral sequences. Now, define the operation from $E^r_{s,t}$ to $E^r_{s+1,t}$ by
	\[\delta^r_{n+1}:=(-1)^t\rho^r\mathcal B^r\eta^r.\]
	Note that $\delta^r_{n+1}$ is well-defined, as $H(F_s(n+1))\subseteq F_s(n+1)$. Moreover, since $d\mathcal B+\mathcal B d=0$ and $\eta^r$ and $\rho^r$ commute with the differentials, we have that
	\begin{equation*}
		\begin{split}
			 d^r\delta^r_{n+1}(x) & =(-1)^t\rho^r d^r\mathcal B^r\eta^r(x)\\
			 & =(-1)^{t+1}\rho^r\mathcal B^r\eta^r(d^r x)\\
			 &=(-1)^r\delta^r_{n+1}(d^r x)
		\end{split}
	\end{equation*}
	for $x\in E^r_{s,t}(n+1)$. In particular, $\delta^r_{n+1}$ induces $\delta^{r+1}_{n+1}$ in the next page.
	
	To identify $\delta^2_{n+1}$, we first take homology for $d^{int}$ for each fixed $i$. By Equation \eqref{eq:dH}, $c_a$ acts on this homology as the identity map. Hence, on the homotopy groups of the bar terms, $F_i$ fixes $\Ff[\mu]$ and cyclically permutes tensor factors in $A_n$. Therefore, $b$ is precisely the differential in the Hochschild complex, and $B_{G_{\bullet}}$ corresponds to Connes' opertor. Since face maps, degeneracy maps and $F_i$ are pointed, they commutes with the unit map and $\Ff$-action map. Again, since $H(F_s)\subseteq F_s$, it is killed in $E^2$-page. It follows that $\delta^2_{n+1}=id_{\Ff[\mu]}\otimes B$ by \cite[\S2.1]{Lod92}. It remains to compute Connes' operator on $E^2_{*,*}(n+1)\cong \Cc.$ If $n+1\notin U$, then we can reduce to compute Connes' operator on the algebra $A$ as in Lemma \ref{lem:HH(A)} by Equation \eqref{eq:E2}, which follows from \cite[Equation (3)]{Mat22}, \cite[Lemma 4]{Spe20} and \cite[Example 2.9]{KN18}, respectively. If $n+1\in U$, then $x_{U,u}$ comes from the $\mathrm{Tor}$-factor, and the formula follows from the same references as above. Finally, $\mu\in B_{n+1}$ is also killed, since $\delta^2_{n+1}=id_{\Ff[\mu]}\otimes B$ and $B(1)=0$.
\end{proof}

\section{Vanishing of differentials}\label{sec4}
In this section, we proof that all differentials 
\[d^r\colon E^r_{s,t}(n+1)\longrightarrow E^r_{s-r,t+r-1}(n+1)\]
vanishes for $r\geq2$. This implies that $A_{n+1}\cong \Cc$ as $\Ff$-modules. Multiplicative extensions will be treated in the next section. Recall that there is an isomorphism
\[E^2_{*,*}(n+1)\cong \Cc\cong \bigotimes_{\varnothing\neq U\subseteq I_{n+1}}B_U.\]
We consider $\Cc$ as a bigraded algebra under this isomorphism and $I_n$ as the subset of $I_{n+1}$ by omitting $n+1$. By Lemma \ref{lem:HH(A)}, an element is of filtration degree $0$ if and only if it belongs to the subalgebra $B_{\{n+1\}}\otimes\bigotimes_{\varnothing\neq U\subseteq I_n}B_U$. Moreover, for $n+1\in U$ and $|U|\geq2$, elements in $U$-factor have filtration degrees given by
\[\mathrm{fil}(\varrho w)=1,\qquad \mathrm{fil}(\varrho^iw)=p^i,\qquad \mathrm{fil}(\varphi^iw)=2p^i,\]
for $i\geq0$, where $w\in B_{U\setminus\{n+1\}}.$ Also recall that we use the symbol $x_{V,v}$ to denote the element of $\Cc\subseteq E^2_{*,*}(n+1)$ whose $V$-factor is $v\in B_V$ and other factors are $1$, as in the proof of Lemma \ref{lem:suspn+1}. For later use, we introduce the following order on nonempty proper subset of $I_{n+1}$.

\begin{Def}
	We order nonempty proper subset of $I_{n+1}$ by the following rules
	\begin{enumerate}
		\item Subsets containing $n+1$ precede those not containing $n+1$.
		\item Within each of two groups, subsets of larger cardinality precede those of smaller cardinality.
		\item Among subsets in the same group with equal cardinality, choose an arbitrary order.
	\end{enumerate}
	For a monomial $m=\prod_{\varnothing\ne V\subseteq I_{n+1}}x_{V,b_V}$ , we define its weight to be 
	\[\omega(m)=\bigl(|b_V|\bigr)_{\varnothing\ne V\subsetneq I_{n+1}},\]
	where $|b_V|$ is the total degree of $b_V$ and $|1|=0$. Consequently, we can compare monomials by the lexicographic order on their weights
\end{Def}

We also need to choose a basis of $\Cc$.
\begin{Def}
	A monomial $m=\prod_{\varnothing\ne V\subseteq I_{n+1}}x_{V,b_V}$ is called a basis monomial if 
	\begin{enumerate}
	\item when $|V|=1$, $b_V=\mu^i$ for $i\geq0$,
	\item when $|V|\geq2$, $b_V$ is either $1$ or a finite product of the chosen exterior generators and $\gamma_{p^i}^j(w)$ of the chosen divided-power factors $\Gamma(w)$ of $B_V$, for $i\geq0$ and $1\leq j\leq p-1$.
	\end{enumerate}
\end{Def}

\begin{lem}\label{lem:part}
	Let $0\neq y\in (\Cp)_{s,t}$ for some $(s,t)$ with $s+t$ odd and suppose that
	\[\Delta^2_{n+1,i}(y)=y\otimes1+1\otimes y,\qquad \delta^2_{n+1}(y)=0\]
	for $1\leq i\leq n+1$. Among basis monomials occurring in $y$, one of maximal weight has the form
	\[m=\prod_{V\in \Pi}x_{V,b_V},\]
	where $\Pi$ is a partition of $I_{n+1}$ into nonempty proper subsets and every $b_V$ is primitive as an element in $B_V$. The unique subset $U$ containing $n+1$ satisfies $|U|\geq2$. In particular, $b_U$ begins with $\varrho$ or $\varrho^0$, and $\mathrm{fil}(y)=1$. Moreover, in every basis monomial occurring in $y$, the factor $x_{V,b_V}$ equals $1$ whenever $U\subsetneq V\subsetneq I_{n+1}$.
\end{lem}
\begin{proof}
	By Proposition \ref{prop:pinch}, the restriction of $\Delta^2_{n+1,i}$ to $\Cp$ takes values in
	\begin{equation}\label{eq:tariso}
		\Cp\otimes_{C_{I_{n+1}\setminus\{i\}}}\Cp\cong C_{I_{n+1}\setminus\{i\}}\otimes\bigotimes_{i\in V\subsetneq I_{n+1}}(B_V\otimes B_V).
	\end{equation}
	Under this isomorphism, for a basis monomial $M$ in $\Cp\otimes_{C_{I_{n+1}\setminus\{i\}}}\Cp$, write $a_{M,V}$ for its factor at $V$-factor with $V$ not containing $i$, and $b_{M,V}\otimes c_{M,V}$ for its two factors at $V$-factor with $V$ containing $i$. Define the weight of $M$ by $\omega_i(M)=(d_{M,V})_{\varnothing\ne V\subsetneq I_{n+1}}$, where
	\[d_{M,V}=\begin{cases}
	|a_{M,V}|,&i\notin V,\\
	|b_{M,V}|+|c_{M,V}|,&i\in V.
	\end{cases}\]
	Using the lexicographic order, we claim that every basis monomial occurring in $\Delta^2_{n+1,i}(m)$ has weight at most $\omega(m)$. By multiplicativity of $\Delta^2_{n+1,i}$ and additivity of weights, it suffices to consider the case $m=x_{V,v}$. If $i\notin V$, then $\Delta^2_{n+1,i}(x_{V,v})=1\otimes x_{V,v}$ whose weight is the same as $x_{V,v}$. If $i\in V$, naturality implies that only $W$-factors can occur in $\Delta^2_{n+1,i}(x_{V,v})$. Writing $\Delta_{B_V}(v)=\sum v'\otimes v''$, we obtain that
	\[\Delta^2_{n+1,i}(x_{V,v})=\sum x_{V,v'}\otimes x_{V,v''}+x\]
	for some $x\in\mathrm{Ker}(Q^2_{n+1,i})$ by Lemma \ref{lem:prop of delta}.(4). Hence every basis monomial $M$ in $x$ has a positive degree factor in some $W$-factor with $W\subsetneq V$. That is 
	\[d_{M,V}+\sum_{\varnothing\neq W\subsetneq V} d_{M,W}=|v|\]
	with $d_{M,W}>0$ for some $W\subsetneq V$. Hence $\omega_i(M)<\omega(x_{V,v})$ by definition. 
	
	Now, write $y=\sum_m c_m m$ with $m$ basis monomials. Using the lexicographic order, let $(d_W)_W$ be maximal among $\omega(m)$ and set 
	\[z=\sum_{\omega(m)=(d_W)_W}c_m m\ne0.\]
	If $d_W=0$ for every $W$ containing $i$, then $z\in C^{\prime}_{I_{n+1}\setminus\{i\}}$ and $\Delta^2_{n+1,i}(z)=z\otimes 1$. By the above argument and maximality of $(d_W)_W$, comparing $(d_W)_W$ weight parts in $\Delta^2_{n+1,i}(y)$ which implies that $z=2z$, a contradiction. If more than one $W$-factors contain $i$ with $d_W>0$, we choose some $W_0$-factor with $d_{W_0}>0$. Moreover, using Equation \eqref{eq:tariso}, we define
	\[P_W:=\begin{cases}
	p_{W,d_W}\otimes p_{W,0}, & W=W_0,\\
	p_{W,0}\otimes p_{W,d_W}, & W\ne W_0,
	\end{cases}\]
	for $i\in W$, and
	\[P=\left(\bigotimes_{i\notin W}p_{W,d_W}\right)\otimes\left(\bigotimes_{i\in W}P_W\right),\]
	where $p_{W,a}$ be the projection of $B_W$ onto its component of total degree $a$. Since every basis monomial occurring in $\Delta^2_{n+1,i}(m)$ has weight at most $\omega(m)$, we deduce that
	\begin{equation}\label{eq:y-z}
		P\circ \Delta^2_{n+1,i}(y-z)=0.
	\end{equation}
	Since there is another $W$ containing $i$ with $d_W>0$, we have that $P(z\otimes1)=P(1\otimes z)=0$. Hence
	\begin{equation}\label{eq:y+y}
		P(y\otimes1+1\otimes y)=P(z\otimes1+1\otimes z)=0.
	\end{equation}
	By the counit identity, we have that $(p_{W,d_W}\otimes p_{W,0})\circ\Delta_{B_W}(b)=b\otimes1$ and $(p_{W,0}\otimes p_{W,d_W})\circ\Delta_{B_W}(b)=1\otimes b$, for $b\in(B_W)_{d_W}$. Hence, on the $(d_W)_W$ weight, $P\circ\Delta^2_{n+1,i}$ retains every factor and only inserts units and reorders the factors, which induces an isomorphism onto a component of $\Cp\otimes_{C_{I_{n+1}\setminus\{i\}}}\Cp$. Hence $P\circ\Delta^2_{n+1,i}(z)\neq 0$, which contradicts Equations \eqref{eq:y-z} and \eqref{eq:y+y}. It follows that there is exactly one $W$ containing $i$ with $d_W>0$. Hence, for a fixed $i$, we can write $z=\sum a_{\alpha}\otimes b_{\alpha}$ with $a_{\alpha}\in \bigotimes_{\Pi\setminus\{W\}}B_V$ and $b_{\alpha}\in B_W$, where $W\in\Pi$ contains $i$. Consequently, we obtain that
	\[\Delta^2_{n+1,i}(z)=\sum a_{\alpha}\otimes\Delta^2_{n+1,i}(b_{\alpha}).\]
	Again, comparing the maximal weight in $\Delta^2_{n+1,i}(y)$ implies that $\Delta_{B_W}(b_{\alpha})=b_{\alpha}\otimes1+1\otimes b_{\alpha}$. Let $i$ range from $1$ to $n+1$. We obtain that, for every $x_{V,v}$ in $z$, $v$ is primitive in $B_V$. 
	
	Finally, choose the unique $U\in\Pi$ containing $n+1$ and a basis monomial $m=\prod_{V\in \Pi}x_{V,b_V}$ occurring in $z$ with each $b_V$ primitive in $B_V$. If $U=\{n+1\}$, then $|b_U|=|\mu^{p^j}|$ is even. Since the total degree of $y$ is odd, there exist $V\in\Pi\setminus\{U\}$ with $|b_V|$ odd. Set $W=V\cup \{n+1\}$ and consider
	\[w=x_{W,\sigma_{B_V}(b_V)}\prod_{T\in\Pi\setminus\{V\}}x_{T,b_T},\]
	where $\sigma_{B_V}$ is considered as a map from $B_V$ to $B_W$. By Remark \ref{rem:phix}, $\sigma_{B_V}(b_V)\neq0$, since $|b_V|$ is odd. We claim that $w$ only occurs in $\delta^2_{n+1}(m)$ after applying $\delta^2_{n+1}$ to basis monomials of $y$. Hence $\delta^2_{n+1}(y)\neq0$, a contradiction. To see this claim, suppose that $m^{\prime}$ is a basis monomial such that $\delta^2_{n+1}(m^{\prime})$ contains $w$. Note that basis monomials in $y$ cannot contain $V$-factors such that $b_V\neq0$, $n+1\in V$ and $|V|\geq2$, since $m$ has the maximal weight. Since $W$ and $U$ are only factors in $w$ containing non-trivial algebra generators and Connes' operator is a derivation, $m^{\prime}$ contains a $V$-factor by Proposition \ref{prop:connes}, say $x_{V,v}$. If $v=ab$ is not a algebra generator with $a,b\neq 1$, then $\delta^2_{n+1}(x_{V,ab})=\delta^2_{n+1}(x_{V,a})x_{V,b}\pm x_{V,a}\delta^2_{n+1}(x_{V,b})$. The $V$-factor in $w$ is $1$, so $a=1$ or $b=1$, a contradiction. This further implies that $v=b_V$ and $m^{\prime}=m$. Therefore, $|U|\geq2$.
	
	It remains to show that $b_U\neq\varphi^0(a)$. Otherwise, since \(b_U\) has even degree and $m$ has odd total degree, we can choose a $V\in\Pi\setminus\{U\}$ such that $b_V$ has odd degree. In particular, $b_V$ is an exterior algebra generator, and $\sigma_{B_V}(b_V)$ is a nonzero algebra generator by Remark \ref{rem:phix}. Set \(W=V\cup\{n+1\}\) and consider the basis monomial
	\[w=x_{W,\sigma_{B_V}(b_V)}\prod_{T\in\Pi\setminus\{V\}}x_{T,b_T}.\]
	By Proposition \ref{prop:connes}, any contribution to the coefficient of $w$ in $\delta^2_{n+1}(y)$ must create a generator in either the $U$-factor or the $W$-factor, as $\delta^2_{n+1}$ is a derivation. Remark \ref{rem:phix} implies that it cannot create the $U$-factor. Consequently, the newly created generator must lie in the $W$-factor, and its source factor must be $V=W\setminus\{n+1\}$. Since the $V$-factor of $w$ is $1$, the original $V$-factor must be an algebra generator. Since the $W$-factor of $w$ is an algebra generator, the original $W$-factor must be $1$. Corollary \ref{cor:QP} and Remark 2.9 then implies that the original $V$-factor to be $b_V$ and all other factors are unchanged. Hence only $m$ contributes, with coefficient $\pm c_m\neq0$, a contradiction.
	
	Therefore, $b_U$ has filtration one, and so is $y$. The last claim follows from $m$ has the maximal weight. 
\end{proof}

We also need the following lemma which considers a general spectral sequence.

\begin{lem}\label{lem:gen}
	Let $E^r_{*,*}$ be a first-quadrant multiplicative spectral sequence whose bidegree of differentials is $(-r,r-1)$ for $r\geq2$. Suppose that $E^2_{0,0}=\Ff$, $E^2_{1,0}=E^2_{0,1}=0$, and every algebra generator of odd total degree has filtration at most $1$. If $D$ is the smallest total degree of a class which supports a nonzero differential and $R$ is the earliest page on which a class of total degree $D$ supports a nonzero differential, then $D$ is even,
	\begin{equation}\label{eq:d-2}
		E^2_{s,t}=E^{\infty}_{s,t}
	\end{equation}
	for $s+t\leq D-2$, and 
	\begin{equation}\label{eq:D}
		E^2_{s,t}=E^R_{s,t}
	\end{equation}
	for $s+t\leq D$. Moreover, no differential enters total degree $D$ on any page.
\end{lem}

\begin{proof}
	Suppose that $D$ is odd, then we can write the class as $x=\sum g_{\alpha}+\sum u_{\beta}v_{\beta}$, where $g_{\alpha}$ is a algebra generator, $|g_{\alpha}|=D$, $|u_{\beta}|+|v_{\beta}|=D$, and $|u_{\beta}||v_{\beta}|>0.$ Since $R\geq 2$ and $\mathrm{fil}(g_{\alpha})=1$, $d^R(g_{\alpha})=0$. By minimality of $D$, $d^R(u_{\beta})=d^R(v_{\beta})=0$. Hence $d^R(x)=0$, a contradiction. Therefore, $D+1$ is odd and the same argument implies that elements of total degree $D+1$ cannot support a nonzero differential. Hence no differential enters total degree $D$. The remaining two equations follows immediately. 
\end{proof}

Since $\Cc$ satisfies the assumption of Lemma \ref{lem:gen}, we choose $D$ and $R$ as above and fix them in the rest of this section.

\begin{lem}\label{lem:ni}
	For $1\leq i\leq n+1$, there is an isomorphism $E^r_{s,t}(n+1,i)\cong E^2_{s,t}(n+1,i)$ for $s+t\leq D$ and $2\leq r\leq R$. 
\end{lem}
\begin{proof}
	It follows from naturality, the proof of Lemma \ref{lem:gen} and the Leibniz rule.
\end{proof}

Now, we can proof that $E^r_{*,,*}(n+1)$ collapses at the $E^2$-page. For each $\varnothing\neq V\subsetneq I_{n+1}$, the inductive step gives a representative $\widetilde{x}_{V,v}\in A_V$ of $v\in B_V$ and the assignment $v\mapsto \widetilde{x}_{V,v}$ is an algebra map. The inclusion $V\to I_{n+1}$ also induces a canonical map $\Phi_V\colon A_V\to A_{I_{n+1}}$. Define the algebra map $\Phi\colon \Cp\to A_{I_{n+1}}$ by
\[x_{V,v}\mapsto \Phi_V(\widetilde{x}_{V,v}).\]

\begin{thm}\label{thm:collapse}
	All differentials $d^r$ in $E^r_{*,*}$ are trivial for $r\geq2$. In particular, $\mathrm{gr}A_{n+1}\cong \Cc$ as $\Ff$-algebras.
\end{thm}
\begin{proof}
	Suppose that there is a nozero differential. First, note that $x_{V,v}$ cannot support a differential for $V\subsetneq I_{n+1}$. Indeed, if $n+1\notin V$ or $V=\{n+1\}$, then $\mathrm{fil}(x_{V,v})=0$. If $\{n+1\}\subsetneq V$, then the inclusion induces a map of spectral sequences $E^r_{*,*}(V)\to E^r_{*,*}(n+1)$. Since $x_{V,v}$ is a permanent cycle in $E^r_{*,*}(V)$, it cannot support a differential in $E^r_{*,*}(n+1)$. In particular, we can choose an algebra generator $x\in B_{n+1}$ such that $d^R(x)=y\neq0$ which is of total degree $D-1$. Write $y=\sum_mc_mm$ with $m$ basis monomials. For a fixed $1\leq i\leq n+1$, we can write $y=\sum_{S_1}c_mm+\sum_{S_2}c_mm$, where $S_1$ contains all $m$ which contains some $x_{V,v}$ with $i\in V$ and $|v|>0$. Collapsing $i$th $S^1$ in $T^{n+1}$ implies that $\bar{x}=\sum_{S_2}c_m\overline{m}=0$. Since $\{\overline{m}\}_{m\in S_2}$ remain linear independent, $c_m=0$ for $m\in S_2$. In other words, for any basis monomial $m$ in $y$, we have that
	\begin{equation}\label{eq:U(m)}
		U(m)=\bigcup_{|b_V|>0}V=I_{n+1},
	\end{equation}
	where $b_v$ is the $V$-factor in $m$. 
	
	Now, the counit identity implies that $\Delta^2_{n+1,i}(x)-x\otimes1-1\otimes x \in K_i\otimes_{C_{I_{n+1}\setminus \{i\}}}K_i$, where $K_i$ is the kernel of $\Cc\to C_{I_{n+1}\setminus \{i\}}$. Applying $d^R$ to both sides implies that $y$ is primitive with respective to $\Delta^2_{n+1,i}$ for every $1\leq i\leq n+1$, by Lemma \ref{lem:ni}. Moreover, Corollary \ref{cor:BncongSn} implies that $q_{n+1}^R(y)=q_{n+1}^R\circ d^R(x)=d^R\circ q^R_{n+1}(x)=0$, which further shows that $q^2_{n+1}(y)=0$. Hence, by Lemma \ref{lem:pdesc}, we have that $y\in \Cp$. On the other hand, since $x$ is an algebra generator in $B_{n+1}$, $\delta^r_{n+1}(x)=0$ for $2\leq r\leq R$, by Proposition \ref{prop:connes}. Hence the same proposition implies $\delta^R_{n+1}(y)=(-1)^Rd^R\circ\delta^R_{n+1}(x)=0$. Equation \eqref{eq:D} shows that $\delta^2_{n+1}(y)=0$. Consequently, Lemma \ref{lem:part} implies that $\mathrm{fil}(y)=1$ and there is a basis monomial $m$ occurring in $y$ with maximal weight, together with a subset $U\subsetneq I_{n+1}$, such that
	\[m=\prod_{V\in \Pi}x_{V,b_V}\]
	and $n+1\in U\in\Pi$, $|U|\geq 2$. In particular $n\geq2$. Otherwise, $U=I_2$, a contradiction. Note that, by the choice, the filtration of $x$ is $\mathrm{fil}(y)+R=R+1$ and belongs to $\{1,2\}\cup\{p^j,2p^j\ |\ j\geq1\}$. We claim that there is an equivalence
	\begin{equation}\label{eq:2=infty}
		E^2_{s,D-s}(n+1)\cong E^{\infty}_{s,D-s}(n+1),
	\end{equation}
	for $0\leq s\leq2$. Indeed, if $R>2$, then it follows from Lemma \ref{lem:gen}. Next, assume that $R=2$. Then $\mathrm{fil}(x)=3$ and $p=3$. Hence $x=\varrho^1(a)=\gamma_3(\varrho_0a)$ with $|a|$ odd, which implies that $D=|x|\equiv0\pmod 3$. Since every algebra generator of $\Cc$ of filtration degree two has total degree $2+3|b|\equiv2\pmod3$, every basis monomial in $E^2_{2,D-2}(n+1)$ must be decomposable. Lemma \ref{lem:gen} then shows that basis monomials in $E^2_{2,D-2}(n+1)$ are permanent cycles, as there is no element of total degree $1$. The case $s=0,1$ follows immediately from Lemma \ref{lem:gen}.
	
	Write $y=\sum c(\prod_Vx_{V,v})$. Then $V\subsetneq I_{n+1}$ and each $x_{V,v}$ comes from an element in $A_V$. Then	 $\Phi(y)= \sum c(\prod_V\Phi_V(\widetilde{x}_{V,v}))$
	has of filtration degree at most one and have associated graded class given by the image of $y$ in $E^{\infty}$-page. Since $y$ is a boundary, $\Phi(y)\in F_0(n+1)$. We show that this is impossible by using suspension operations. 
		
	To prove this, we first consider the suspension operation on
	\[F_0(n+1)=\mathrm{Im}\left(\Ff[\mu]\otimes A_n\longrightarrow A_{n+1}\right).\]
	For $i\neq n+1$, $\sigma_{n+1,i}$ maps $A_n$ to $A_n$. Hence $\sigma_{n+1,i}$ preserves the image of $A_n$ in $F_0(n+1)$, by naturality. Next, let $\widetilde{\lambda}_{n+1,i}=\sigma_{n+1,i}(\mu_{n+1})\in A_{n+1}$ and $\lambda_{n+1,i}=\epsilon x_{\{i,n+1\},\varrho\mu}\in E^2_{1,2}(n+1)$ for $i\neq n+1$, where $\epsilon=\pm1$ is the fixed sign determined by the chosen coordinates. We can also consider $\lambda_{n+1,i}$ as the image of the same named element in $A_{\{i,n+1\}}$. Then $\widetilde{\lambda}_{n+1,i}=\Phi_{\{n+1,i\}}(\lambda_{n+1,i})\in F_1(n+1)$, since there is no nonzero element in $F_0(n+1)\cap\Phi_{\{n+1,i\}}(A_{\{i,n+1\}})$ having total degree $3$. Since every element of $F_0(n+1)$ is a sum of $\mu_{n+1}^kz\in A_{n+1}$ with $z\in A_n$ considered as an element of $A_{n+1}$ and 
	\[\sigma_{n+1,i}(\mu_{n+1}^kz)=k\mu_{n+1}^{k-1}\widetilde{\lambda}_{n+1,i}z+\mu_{n+1}^k\sigma_{n+1,i}(z)\in F_1(n+1),\]
	we deduce that $\sigma_{n+1,i}(F_0(n+1))\subseteq F_1(n+1)$ for $i\neq n+1$.
	
	Next, we consider generators of filtration one in $V$-factors with $n+1\in V\subsetneq I_{n+1}$. For such a $V$ and a generator $b=\sigma_{B_{V\setminus\{n+1\}}}(a)$ (see Remark \ref{rem:phix}), we have that
	\[\widetilde{x}_{V,b}=\pm\sigma_{n+1,n+1}(\widetilde{x}_{V\setminus\{n+1\},a)})+z\]
	with $z\in F_0(n+1)$. Indeed, by Lemmas \ref{lem:HH(A)} and \ref{lem:susp}, these two classes have the same class $x_{V,b}$ of filtration one on the associated graded. Hence their difference lies in $F_0(n+1)$. Hence
	\begin{equation*}
		\begin{split}
			\sigma_{n+1,i}(\widetilde{x}_{V,b}) & =\pm\sigma_{n+1,i}\circ\sigma_{n+1,n+1}(\widetilde{x}_{V\setminus\{n+1\},a)})+\sigma_{n+1,i}(z)\\
			& =\mp \sigma_{n+1,n+1}\circ\sigma_{n+1,i}(\widetilde{x}_{V\setminus\{n+1\},a)})+\sigma_{n+1,i}(z)
		\end{split}
	\end{equation*}
	for $i\neq n+1$, as interchanging two circles induces multiplication by $-1$. By the above argument, $\sigma_{n+1,i}(z)\in F_1(n+1)$ and $\sigma_{n+1,i}(\widetilde{x}_{V\setminus\{n+1\},a)})\in F_0(n+1)$. Therefore, Lemma \ref{lem:susp} shows that $\sigma_{n+1,n+1}\circ\sigma_{n+1,i}(\widetilde{x}_{V\setminus\{n+1\},a)})\in F_1(n+1)$, which further implies that $\sigma_{n+1,i}(\widetilde{x}_{V,b})\in F_1(n+1)$. 
	
	We have seen that $y\in \Cp$. Hence we can set
	\begin{equation}\label{eq:y as series}
		y=\sum_{a\geq0}\mu_{n+1}^ay_a,
	\end{equation}
	with $y_a\in \Cp$ not involving $\mu_{n+1}$, where $\mu_{n+1}$ denotes the image of $\mu\in B_{\{n+1\}}$ in $E^r_{*,*}(n+1)$. Note that every nonzero $y_a$ has filtration degree one. Consider the element
	\[Y=\sum_{a\geq0}\mu_{n+1}^a\sigma_{n+1,i}(\Phi(y_a))\in (A_{n+1})_D.\]
	Since $\mathrm{fil}(y_a)=1$ and $y_a$ do not involve $\mu_{n+1}$, every monomial of $y_a$ has exactly one filtration degree one generator in a factor containing $\{n+1\}$ properly. Since $\sigma_{n+1,i}$ is a derivation, the above argument implies that $Y\in F_1(n+1)$. Since $|\mu_{n+1}|$ is even, applying $\sigma_{n+1,i}\circ\Phi$ to Equation \eqref{eq:y as series} gives
	\[\sigma_{n+1,i}\circ\Phi(y)=Y+\widetilde\lambda_{n+1,i}\Phi(w)=Y+\Phi(\lambda_{n+1,i}w),\]
	where $w=\sum_{a\geq1}a\mu_{n+1}^{a-1}y_a$. Since $\Phi(y)\in F_0(n+1)$ and $\sigma_{n+1,i}(F_0(n+1))\subseteq F_1(n+1)$, we see that $\Phi(\lambda_{i,n+1}w)\in F_1(n+1)$. Since $\Phi(\lambda_{i,n+1}w)$ represents $\lambda_{n+1,i}w$ in $E^{\infty}_{2,D-2}$, we deduce that $\lambda_{n+1,i}w=0$ in $E^2_{2,D-2}(n+1)$ by Equation \eqref{eq:2=infty}. In particular, we have that
	\begin{equation}\label{eq:s=y}
	\sigma_{n+1,i}\circ\Phi(y)=Y.		
	\end{equation}
	We compare the image of both sides in $E^{\infty}_{1,D-1}\cong F_1(n+1)/F_0(n+1)$. Since $y\in \Cp$, Equation \eqref{eq:U(m)} implies that every basis monomial of $y$ has at least two factors. For each such a factor $b$, $|b|\leq(D-1)-2=D-3,$ since  every positive degree is at least two. Hence $|\sigma_{n+1,i}(b)|=|b|+1\leq D-2$, if $\sigma_{n+1,i}(b)\neq0$. Hence these factors and their images under suspension maps lie in $E^2_{s,t}(n+1)\cong E^{\infty}_{s,t}(n+1)$ for $s+t\leq D-2$, by Equation \eqref{eq:d-2}. By multiplicativity, we compute the following coefficient in $E^2_{1,D-1}(n+1)\cong E^{\infty}_{1,D-1}(n+1)$. Choose $i\notin U$ and consider the coefficient of the monomial
	\[t=x_{W,\,\sigma_{B_U}(b_U)}\prod_{V\in\Pi\setminus\{U\}}x_{V,b_V}\]
	with $W=U\cup\{i\}$. Write $i^W_U\colon T^U\to T^W$ for the inclusion and $\sigma_{W,i}$ for the suspension operation on $T^W\otimes\Ff$. Corollary \ref{cor:QP}, Remark \ref{rem:phix} and Lemma \ref{lem:part} implies that
	\[(q_W)_*\circ\sigma_{W,i}\circ(i^W_U)_*(\widetilde{x}_{U,b_U})=\pm\sigma_{B_U}(b_U)\neq 0.\]
	The same remark and lemma also imply that $\sigma_{B_U}(b_U)$ has filtration one in $\pi_*(S^{n+1}\otimes\Ff)$. Moreover, $F_0(S^{n+1})=\Ff$, so $\sigma_{B_U}(b_U)$ represents a nonzero class in $E^2_{*,*}(S^{n+1})$. Since $\sigma_{n+1,i}(m)\in F_1(n+1)$ and $m$ involve no $\mu_{n+1}$,  $m$ contributes $\pm c_m\neq0$ to the coefficient of $t$ in $E^2_{1,D-1}(n+1)$ in the right of Equation \eqref{eq:s=y}. Moreover, there is no other contribution to the coefficient of $t$ in $Y$. To see this, we only need to consider monomials in $\sigma_{n+1,i}\sigma(\Phi(y_0))$. By Lemma \ref{lem:part}, basis monomials in $y_a$ connot contain a $W$-factor with $W\subsetneq I_{n+1}$, since $n+1\in U\subsetneq W$. Moreover, if $W= I_{n+1}$, then no basis monomial in $y$ can contain a $W$-factor, since $y_a\in \Cp$. Hence, to contribute to a coefficient of $t$, the basis monomial must contain a $K$-factor with $K\cup\{i\}=W$, as $\sigma_{n+1,i}$ is a derivation which maps $K$-factor to $K\cup\{i\}$-factor. Since $i\notin U$, $U\subseteq K$. The maximality of $\omega(m)$ then implies that $U=K$ and the $K$-factor $x_{K,k}$ has degree at most $|b_U|$. Therefore, if it contributes to the coefficient of $t$, then we have that
	\[|b_U|+1\leq|k|+1\leq|b_U|+1.\]
	That is, $|k|=|b_U|>0$. By Lemma \ref{lem:part}, $\mathrm{fil}(b_U)=1$. Hence the filtration degree of $k$ is also one, since $|U|\geq2$ and $\mathrm{fil}(y_a)=1$. Since positive degree elements in $B_U$ have filtration degrees at least one and the filtration degree of nonzero product equals the sum of filtration degrees, $b_U$ and $k$ are indecomposable. Corollary \ref{cor:QP} and Remark \ref{rem:phix} imply that $b_U=k$. Consequently, the other factors are the same. We deduce that $\overline{Y}\in E^2_{1,D-1}(n+1)$ contribute $\pm c_m$ to the coefficient of $t$. On the other hand, since $\Phi(y)\in F_0(n+1)$, we can write $\Phi(y)$ as $\Phi(y)=\sum_{a\geq0}\mu_{n+1}^a\widetilde{y}^{\prime}_a,$ where $y^{\prime}_a\in C_{I_n}$ and $\widetilde{y}^{\prime}_a$ is viewed as an element in $A_{n+1}$ via the canonical inclusion. Images of $\widetilde{y}^{\prime}_a$ under $\sigma_{n+1,i}$ still omit $n+1$ for $1\leq i\leq n$. Hence they have filtration zero. Therefore, only $\sigma_{n+1,i}(\mu_{n+1}^a)$ contributes to filtration one, which implies that
	\[\overline{\sigma_{n+1,i}\circ\Phi(y)}=\lambda_{n+1,i}\sum_{a\geq1}a\mu_{n+1}^{a-1}y^{\prime}_a.\]
	This cannot contributes to the coefficient of $t$, a contradiction. Thus all differentials vanish.
\end{proof}

\section{Multiplicative extensions}\label{sec5}
By Section \ref{sec4}, the spectral sequence $E^r_{*,*}(n+1)$ collapses at the $E^2$-page. Consequently, $\mathrm{gr}A_{n+1}\cong \Cc$ as $\Ff$-algebras. In this section, we resolve multiplicative extensions and complete the induction.

Collapsing $i$th circle in $T^{n+1}$ induces a map $c_i\colon A_{n+1}\to A_{I_{n+1}\setminus\{i\}}$ whose section is $\Phi_{I_{n+1}\setminus\{i\}}$ defined in Section \ref{sec4}.

\begin{prop}\label{prop:free}
	The map $\Phi$ is injective. For any monomial basis $v$ of $B_{n+1}$, there is a lift $\widetilde{v}\in A_{n+1}$ such that $\tilde{1}=1$  and $q_{n+1}(a\widetilde{v})=\epsilon(a)v$, where $a\in\Cp$ and $\epsilon$ is the augmentation of $\Cp$. Moreover, $c_i(\widetilde{v})=0$ for $v\neq1$ and $i\in I_{n+1}$,
	\[A_{n+1}\cong\bigoplus_v\Cp\widetilde{v},\]
	and $A_{n+1}$ is a free $A_U$ module for $U\subsetneq I_{n+1}$.
\end{prop}
\begin{proof}
	Considering $\Cp$ as a subalgebra of $\Cc$ equips $\Cp$ with a bigraded algebra structure. We can further define a filtration on $\Cp$ by $F_i(\Cp):=(\Cp)_{\leq i,t}.$ Then $\Phi$ is a filtered map and $\mathrm{gr}(\Phi)\colon \Cp\cong \mathrm{gr}(\Cp)\to\Cc$ with $\mathrm{gr}(\Phi(x))=x\otimes1\in \Cp\otimes B_{n+1}$ by its definition. Hence $\Phi$ is injective. In particular, we can identify $\mathrm{Im}(\Phi)$ with $\Cp$. 
	
	Next, take the lift of $1$ to be $1$. For $v\neq1$, since $q_{n+1}^2=\epsilon\otimes id\colon\Cp\otimes B_{n+1}\to B_{n+1}$, we choose a lift $x$ of each monomial basis $v$ of $B_{n+1}$ in $A_{n+1}$ whose image in $E^{\infty}_{*,*}(n+1)$ is $1\otimes v$. Hence $q_{n+1}(x)=v+v^{\prime}$ and $\mathrm{fil}(v^{\prime})<\mathrm{fil}(v)$. Choose a lift $x^{\prime}$ of $v^{\prime}$ whose image in $E^{\infty}_{*,*}(n+1)$ is $1\otimes v^{\prime}$. We have that $q_{n+1}(x-x^{\prime})=v+v^{\prime\prime}$ and $\mathrm{fil}(v^{\prime\prime})<\mathrm{fil}(v^{\prime})$. Repeating this process, we obtain a lift $\bar{v}$ of $v$. Set $v_1=\bar{v}$ and inductively define
	\[v_{i+1}:=(id-\Phi_{I_{n+1}\setminus\{i\}}\circ c_i)(v_i)\]
	for $1\leq i\leq n+1$. Let $\widetilde{v}=v_{n+2}$. This is a lift of $v$, as $q_{n+1}\circ\Phi_{I_{n+1}\setminus\{i\}}(x)=0$ for $|x|>0$. Moreover, for $1\leq j<k\leq n+1$, we have that
	\begin{equation*}
		\begin{split}
			c_j\circ (id-\Phi_{I_{n+1}\setminus\{k\}}\circ c_k) & =c_j-i_{j,k}\circ c_j\circ c_k\\
			& =c_j-i_{j,k}\circ c_k\circ c_j
		\end{split}
	\end{equation*}
	where $i_{j,k}$ is the canonical inclusion $A_{I_{n+1}\setminus\{j,k\}}\to A_{n+1}.$ Hence $c_i(\tilde{v})=0$ since $c_i(v_{i+1})=0$. The equiation $q_{n+1}(a\widetilde{v})=\epsilon(a)v$ follows from the equation $q_{n+1}(a)=\epsilon(a)$.
	
	Finally, the multiplication induce a map from $\bigoplus_v\Cp\widetilde{v}$ to $A_{n+1}$ whose associated graded is the identity. It is injective by the same argument as in the first paragraph. The method used in the second paragraph shows that it is surjective. The last statement follows from the isomorphism $A_U\cong C_U$ by inductive hypothesis.
\end{proof}

Next, set $\overline\Delta_{n+1,i}(y)=\Delta_{n+1,i}(y)-y\otimes1-1\otimes y$. We have the following form of Lemma \ref{lem:pdesc}.

\begin{lem}\label{lem:hdesc}
	If $y\in A_{n+1}$ satisfies
	\[q_{n+1}(y)=0\qquad \mathrm{and} \qquad \overline\Delta_{n+1,i}(y)\in \Cp\otimes_{C_{I_{n+1}\setminus\{i\}}}\Cp\]
	for $1\leq i\leq n+1$, then $y\in C'_{I_{n+1}}$.
\end{lem}
\begin{proof}
	By Proposition \ref{prop:free}, we can write $y=\sum_v a_v\widetilde v$. Since $\overline\Delta_{n+1,i}(y)\in \Cp\otimes_{C_{I_{n+1}\setminus\{i\}}}\Cp$, the coefficient of $1\otimes\widetilde{v}$ is zero in $\overline\Delta_{n+1,i}(y)$ for $v\neq1$. The rest of proof is the same as that of Lemma \ref{lem:pdesc}.
\end{proof}

Proposition \ref{prop:free} also verifies the hypotheses of Lemma \ref{lem:suspn+1}, which gives the following lemma.

\begin{lem}\label{lem:ker}
	Under the identification of Proposition \ref{prop:free}, there is an isomorphism
	\[\bigcap_{i\in I_{n+1}}\mathrm{Ker}\bigl(\sigma_{n+1,i}\colon C'_{I_{n+1}}\to A_{n+1}\bigr)= P_{I_{n+1}}^p.\]
\end{lem}

Now, we can resolve multiplicative extensions and prove our main theorem.

\begin{proof}[Proof of Theorem \ref{thm:main}]
	By Proposition \ref{prop:free}, the multiplication in $\Cp\cong\mathrm{Im}(\Phi)$ is already determined. Next, we choose lifts of algebra generators of $B_{n+1}$ which represent compatible classes in the $E^2$-page (see Remark \ref{rem:comp}) and satisfy the desired relations, by induction on the degree of algebra generators of $B_{n+1}$. At the degree $d$ step, we assume that the lifts of all algebra generators of degree less than $d$ satisfy the desired relations in $A_{n+1}$, and that there is an $\Ff$-algebra map 
	\[\Phi_{\leq d-1}\colon\Cp\otimes B_{n+1}^{\leq d-1}\longrightarrow A_{n+1}\]
	which is an isomorphism in less than $d$, where $B_{n+1}^{\leq d-1}$ is the subalgebra of $B_{n+1}$ generated by algebra generators of degree less than $d$.
	
	For the initial step, we can choose a lift $\widetilde{v}$ for an algebra generator $v$ of minimal degree $d_0$ by Proposition \ref{prop:free}. If $d_0$ is odd, then $\widetilde{v}^2=0$, as $p$ is odd. Next, suppose that $d_0$ is even. Since $c_i$ admits a section, there is an isomorphism $A_{n+1}\cong A_{I_{n+1}\setminus\{i\}}\oplus\mathrm{Ker}(c_i)$ of $A_{I_{n+1}\setminus\{i\}}$-modules. In particular, we have an isomorphism
	\begin{equation*}
		\begin{split}
			& A_{n+1}\otimes_{A_{I_{n+1}\setminus\{i\}}}A_{n+1}\\
			& \cong A_{I_{n+1}\setminus\{i\}}\oplus(\mathrm{Ker}(c_i)\otimes_{A_{I_{n+1}\setminus\{i\}}}\mathrm{Ker}(c_i).)\\
			& \oplus (\mathrm{Ker}(c_i)\otimes_{A_{I_{n+1}\setminus\{i\}}}A_{I_{n+1}\setminus\{i\}})\oplus (A_{I_{n+1}\setminus\{i\}}\otimes_{A_{I_{n+1}\setminus\{i\}}}\mathrm{Ker}(c_i)).
		\end{split}
	\end{equation*}
	Moreover, the counit identity implies that
	\begin{equation*}
		\begin{split}
			& (c_i\otimes id)\circ\overline{\Delta}_{n+1,i}(\widetilde{v})=\widetilde{v}-c_i(\widetilde{v})-\widetilde{v}=0,\\
			& (id\otimes c_i)\circ\overline{\Delta}_{n+1,i}(\widetilde{v})=\widetilde{v}-\widetilde{v}-c_i(\widetilde{v})=0,
		\end{split}
	\end{equation*}
	as $c_i(x)=0$. The proceeding isomorphism implies that $\overline{\Delta}_{n+1,i}(\widetilde{v})\in \mathrm{Ker}(c_i)\otimes_{A_{I_{n+1}\setminus\{i\}}}\mathrm{Ker}(c_i)$. The minimality of $d_0$ then implies that $\overline{\Delta}_{n+1,i}(\widetilde{v})\in \Cp\otimes_{A_{I_{n+1}\setminus\{i\}}}\Cp$, as $(\mathrm{Ker}(c_i))_0=0$. Since $d_0$ is even, the same holds for $\overline{\Delta}_{n+1,i}(\widetilde{v}^p).$ Moreover, $q_{n+1}(\widetilde{v}^p)=v^p=0$. Therefore, Lemma \ref{lem:hdesc} implies that $\widetilde{v}^p\in\Cp$. Since $\sigma_{n+1,i}$ is a derivation, $\sigma_{n+1,i}(\widetilde{v}^p)=0$, which implies that $\widetilde{v}^p\in P_{I_{n+1}}^p$ by Lemma \ref{lem:ker}. Choose a unique $f_v\in P_{I_{n+1}}$ such that $f_v^p=\widetilde{v}^p$ and set the lift of $v$ to be $\widetilde{v}-f_v$. Since $f_v$ have filtration degree $0$, this lift represents the same class as $v$ in the associated graded. Moreover, $q_{n+1}(f_v)=0$. Hence $q_{n+1}(\widetilde{v}-f_v)=v$. This gives an $\Ff$-algebra map $\Phi_{\leq d_0}\colon\Cp\otimes B_{n+1}^{\leq d_0}\to A_{n+1}$. We need to show that $\Phi_{\leq d_0}$ is an isomorphism in degree $d_0$. By Proposition \ref{prop:free}, there is an isomorphism of $\Ff$-modules $(A_{n+1})_{d_0}\cong(\Cp)_{d_0}\oplus\bigoplus_{|v|=d_0}\Ff\{\widetilde{v}\}.$ The composition $\Psi_d\circ(\Phi_{\leq d_0})_d^{-1}$ maps $1\otimes v$ to $\widetilde{v}-f_v$, which is clearly an isomorphism. Hence $(\Phi_{\leq d_0})_d$ is also an isomorphism. Moreover, $c_i(f^p)=c_i(\widetilde{v}^p)=0$, which implies that $c_i(f)=0$, as its image is a polynomial. Therefore, $c_i(\widetilde{v}-f)=0$.
	
	At the degree $d$ step, assume that the lifts of algebra generator of degree less that $d$ satisfy the desired relations and give the $\Ff$-algebra isomorphism $\Cp\otimes B_{n+1}\cong A_{n+1}$ under degree $d$. Let $v$ be an algebra generator of degree $d$ and choose a lift $\widetilde{v}$ as in Proposition \ref{prop:free}. Again, if $d$ is odd, then $\widetilde{v}^2=0$. Therefore, we assume that $d$ is even. Then the same argument as the initial step implies that $\overline{\Delta}_{n+1,i}(\widetilde{v})\in \mathrm{Ker}(c_i)\otimes_{A_{I_{n+1}\setminus\{i\}}}\mathrm{Ker}(c_i)$. The equation $(\mathrm{Ker}(c_i))_0=0$ shows that degrees of factors in both $\mathrm{Ker}(c_i)$ are less than $d$. By the inductive hypotheses, only monomials containing only $\mu_j$ can survive after taking $p$th powers. Hence we deduce that $\overline{\Delta}_{n+1,i}(\widetilde{v}^p)\in\Cp\otimes_{A_{I_{n+1}\setminus\{i\}}}\Cp$. Again, the same argument as the initial step implies that we can choose a lift $x$ of $v$ such that $x^p=0$, $x$ represents the same clasee as $v$ in the associated graded, $q_{n+1}(x)=v$, and $c_i(x)=0$. 
	
	After determining lifts of algebra generators less than degree $d$, we can lift the $\Ff$-algebra map $\Phi_{\leq d-1}$ to an $\Ff$-algebra map
	\[\Phi_{\leq d}\colon\Cp\otimes B_{n+1}^{\leq d}\longrightarrow A_{n+1}.\]
	To complete the inductive step, it remains to show that $(\Phi_{\leq d})_d$ is an isomorphism. By our choice, the restriction of $q_{n+1}\circ\Phi_{\leq d}$ to $B_{n+1}^{\leq d}$ equals the subalgebra inclusion. Hence Proposition \ref{prop:free} implies that
	\[\Phi_{\leq d}(1\otimes v)=\widetilde{v}+\sum_{|w|<|v|}a_{w,v}\widetilde{w}\]
	with $a_{w,v}\in(\Cp)_{>0}$. The same proposition also gives an isomorphism of $\Ff$-modules
	\[\Psi_d\colon(\Cp\otimes B_{n+1})_d\xrightarrow{\cong}(A_{n+1})_d\]
	with $\Psi_d(c\otimes b)=c\widetilde{b}$. Now, the map $\Psi_d^{-1}\circ(\Phi_{\leq d})_d-id$ maps $1\otimes v$ to $\sum_{|w|<|v|}a_{w,v}\otimes\widetilde{w}$. Note that both $\Phi_{\leq d}$ and $\Psi_d$ are $\Cp$-linear, so we have that
	\begin{equation*}
		\begin{split}
			(\Psi_d^{-1}\circ(\Phi_{\leq d})_d-id)(\sum_{|w|<|v|}a_{w,v}\otimes w) & =\sum_{|w|<|v|}(a_{w,v}\otimes1)\cdot(\Psi_{|w|}^{-1}\circ(\Phi_{\leq d})_{|w|}-id)(1\otimes w)\\
			& = \sum_{|w|<|v|}\left[(a_{w,v}\otimes1)\cdot\sum_{|u|<|w|}(a_{u,w}\otimes\widetilde{u})\right].
		\end{split}
	\end{equation*}
	Iterating this process gives $(\Psi_d^{-1}\circ(\Phi_{\leq d})_d-id)^{\circ (d+1)}=0.$ Denote this map by $F$. We have that
	\[(id+F)\circ\left(\sum_{i=0}^d(-F)^{\circ i}\right)=\left(\sum_{i=0}^d(-F)^{\circ i}\right)\circ(id+F)=id.\]
	Since $(\Phi_{\leq d})_d=\Psi_d\circ(id+F)$, we see that $(\Phi_{\leq d})_d^{-1}=\left(\sum_{i=0}^d(-F)^{\circ i}\right)\circ\Psi_d^{-1}$. In particular, $(\Phi_{\leq d})_d$ is an isomorphism.
\end{proof}

\subsection*{Acknowledgements}
The author is grateful to his advisor Xiangjun Wang for his guidance and support.

\subsection*{Statement on AI use}
This article builds on the author's earlier work, which identified the $E^2$-page with Hochschild homology and computed $\pi_*(T^n\otimes\mathbb F_p)$ for all $n\ge1$ and $p\geq5$, assuming compatibility of Connes' operator with the spectral sequence. The project remained on hold until GPT-6-Astra assisted the author in establishing this compatibility (see Proposition \ref{prop:connes}) and revising several earlier arguments, particularly the proofs of Proposition \ref{prop:pinch} and Lemma \ref{lem:part}. AI assistance was also used in developing the extension to $p=3$, particularly the $R=2$ argument establishing Equation \eqref{eq:2=infty} in the proof of Theorem 4.6, and in language editing. The author takes full responsibility for the mathematical content and the final version of the manuscript.

\end{document}